\documentclass[12pt, english, a4paper]{article}
\usepackage[latin9]{inputenc}
\usepackage[T1]{fontenc}
\usepackage{graphicx}
\usepackage{amsthm}
\usepackage{amsfonts}
\usepackage{mathtools}
    \allowdisplaybreaks
\usepackage{amssymb}
\usepackage{makecell}
\usepackage{comment}
\usepackage{enumitem}
\usepackage{cases}
\usepackage{hyperref}
\usepackage{mathrsfs}
\usepackage{fullpage}
\usepackage{times}
\usepackage[usenames]{color}
\usepackage[dvipsnames]{xcolor}
\usepackage{hyperref}
\usepackage{tikz,tikz-cd}
\usetikzlibrary{decorations.pathreplacing,calligraphy}
\usepackage{extarrows}
\usepackage{float}
\usepackage{ytableau}
\usepackage{xcolor}
\usepackage{circledsteps}
 \usepackage{multirow} \usepackage[dvipsnames]{xcolor}
 \usepackage[svgnames]{xcolor}
\definecolor{lgrey}{rgb}{0.8,0.8,0.8 }
\definecolor{dgrey}{rgb}{0.3,0.3,0.3 }
\definecolor{olivegreen}{rgb}{0.33, 0.42, 0.18}
\definecolor{myseagreen}{RGB}{46, 139, 87}
\newcommand{\blue}[1]{\textcolor{blue}{#1}}
\newcommand{\red}[1]{\textcolor{red}{#1}}

\newcommand{\SSS}{\mathfrak{S}}
\newcommand{\WW}{\mathcal{W}}

\DeclareMathOperator{\St}{\mathrm{Staircase}}

\theoremstyle{plain}
\newtheorem{theorem}{Theorem}

\newtheorem{example}[theorem]{Example}

\newtheorem{lemma}[theorem]{Lemma}

\theoremstyle{definition}
\newtheorem{definition}[theorem]{Definition}

\newtheorem{remark}{Remark}

\usepackage{authblk}

\title{Matchings and shape-Wilf-Equivalence of sets of patterns of length three I: Triples}
 
\author[1]{Sucharita Biswas \thanks{\tt{biswas.sucharita56@gmail.com}}}
\author[2]{Umesh Shankar\thanks{\tt{umeshshankar@outlook.com}}} 
\author[3]{Sivaramakrishnan Sivasubramanian \thanks{\tt{ krishnan@math.iitb.ac.in }}}
\affil[1,3]{Department of Mathematics, Indian Institute of Technology, Bombay Mumbai, 400076, India} 
\affil[2]{Department of Mathematics, Indian Institute of Technology, Madras Chennai, 600036, India}

\date{ }
\begin{document}
\maketitle
\begin{abstract}
    Permutation pattern avoidance on Ferrers boards has become a central topic in enumerative combinatorics with important connections to matchings, set partitions, and other combinatorial structures as it allows one to build families of Wilf-equivalent patterns. While shape-Wilf-equivalence classes have been completely determined for individual patterns and pairs of patterns of length three, the corresponding classification for larger pattern sets has remained open. In this paper, we provide a complete classification of the shape-Wilf-equivalence classes of triples of patterns of length three. Our proofs use a bijective encoding of pattern avoiding transversals to establish all equivalence classes. As an application, we enumerate perfect matchings avoiding triples of patterns of length three for all, extending previous results of Bloom and Elizalde. These enumerative results identify additional families of combinatorial objects counted by the Fuss-Catalan numbers and by other integer sequences appearing in the OEIS.
\end{abstract}
\textbf{\small{Keywords:}}{ permutation patterns, Ferrers boards, pattern-avoiding perfect matchingss, transversals, encoding bijections.}{\let\thefootnote\relax}

\section{Introduction}

Let $\mathbb{N}$ denote the set of all natural numbers. For $n \in \mathbb{N}$, we define $[n] = \{1, 2, \ldots, n\}$ and let $\SSS_n$ be the set of all permutations of the set $[n]$. An occurrence of a \textit{classical pattern} of size $k$, $p = p_1 p_2 \cdots p_k \in \SSS_k$, in a permutation $\pi \in \SSS_n$ is defined by a subsequence $\pi_{i_1} \pi_{i_2} \cdots \pi_{i_k}$ (where $1 \leq i_1 < i_2 < \cdots < i_k \leq n$) that is order-isomorphic to $p$; that is, $\pi_{i_j} < \pi_{i_m}$ if and only if $p_j < p_m$. For instance, the permutation $461352$ contains four occurrences of the pattern $132$, namely the subsequences $465$, $132$, $152$ and $352$. Pattern avoidance in permutations is a well-established field in enumerative combinatorics with significant applications in computer science, computational biology, and cryptography. Beyond permutations, patterns are extensively studied in the context of words, the monograph by Kitaev \cite{kitaev-patterns-words}
has a wealth of information on  several variants of pattern avoidance.

A permutation is called \textit{$P$-avoiding} if it does not contain any pattern belonging to the set $P$. In particular, for a single pattern $p$, a permutation avoids $p$ when no subsequence of it is order-isomorphic to $p$. We denote by $\SSS_n(P)$ the collection of permutations in $\SSS_n$ avoiding every pattern in $P$; when $P=\{p\}$, we simply write $\SSS_n(p)$. Two pattern sets $P_1$ and $P_2$ are said to be \textit{Wilf-equivalent}, written 
$P_1\sim P_2$ if $|\SSS_n(P_1)|=|\SSS_n(P_2)|$
for all integers $n\ge1$. When both sets consist of a single pattern, 
say $P_1=\{p_1\}$ and $P_2=\{p_2\}$, we abbreviate the notation to
$p_1\sim p_2.$

Understanding and enumerating Wilf-equivalent classes has been a central question in this area. One key result in understanding Wilf-equivalence classes is the result of Backelin, West and Xin \cite{bwx-main} which introduces the notion of shape-Wilf-equivalence and provides a way to build large families of Wilf-equivalent classes. 
Shape-Wilf equivalence extends permutation patterns to Ferrers boards.
Throughout this paper, Ferrers boards are represented in \textit{French notation}; 
that is, row lengths weakly increase from top to bottom. However, we index rows 
from bottom to top and columns from left to right, and 
$(i,j)$ denotes the cell located in the $i$-th row and $j$-th column.

Let $F=(F_1,\dots,F_n)$ be a Ferrers board where $F_i \geq F_{i+1}$ for 
$1 \leq i \leq n-1$  be the lengths 
of the rows of $F$. A \textit{transversal} 
of $F$ is a placement of $1$'s and $0$'s in the cells of $F$ such that each row and each column contains precisely one $1$. In order for a Ferrers board $F$ to have any 
traversals at all, $F$ must have the same number of rows and columns and
must also contain the staircase diagram $\St_n = F(n, n - 1, \ldots, 2, 1)$, 
thus $F_1 = n$ and $F_n > 0$.

The collection of all transversals of $F$ is denoted by $\SSS_F$. Permutations 
in $\SSS_n$ can be viewed as special cases of such transversals as follows: $\pi=\pi_1\pi_2\cdots\pi_n\in \SSS_n$ determines a transversal of the $n\times n$ square Ferrers board by placing a $1$ in the cell $(\pi_i,i)$ for every $i$, leaving all remaining cells equal to $0$.
Consider a permutation $\alpha\in \SSS_k$ and let $M_\alpha$ denote its permutation matrix. A transversal $T\in\SSS_F$ is said to \textit{contain} $\alpha$ if there exist row indices
$$r_1<r_2<\cdots<r_k 
\>\> \mbox{ and column indices } \>\>  c_1<c_2<\cdots<c_k$$
for which the submatrix of $T$ determined by these rows and columns coincides with $M_\alpha$, with all corresponding cells lying inside the Ferrers board $F$. If no such choice exists, then $T$ is said to \textit{avoid} the pattern $\alpha$.
For a collection of patterns $P$, we write $\SSS_F(P)$ for the set of transversals of $F$ avoiding every member of $P$. Two pattern sets $P_1$ and $P_2$ are called \textit{shape-Wilf-equivalent}, written
$$P_1\sim_s P_2, \>\> \mbox{   whenever  } \>\> |\SSS_F(P_1)|=|\SSS_F(P_2)|$$
for every Ferrers board $F$. In particular, shape-Wilf-equivalence immediately implies Wilf-equivalence for permutations 
by just considering the square Ferrers board. 

A central result of Backelin, West and Xin  \cite[Proposition 2.2]{bwx-main} 
shows that the increasing pattern $12\ldots k$ and the decreasing pattern 
$k(k-1)\ldots 1$ are shape-Wilf-equivalent. This result was later given a 
more direct combinatorial proof by Krattenthaler \cite{growth-diagram} using 
growth diagrams.  This proof also showed 
that $k$-non nesting and $k$-non crossing perfect matchingss are equinumerous. 
Backelin, West and Xin in \cite{bwx-main} also showed that $123 \sim_s 213$.

The classification of individual patterns of length $3$ was completed by Stankova and West \cite{new-wilf}, who showed that $231 \sim_s 312$. This shows that there are exactly three shape-Wilf-equivalence classes in $\SSS_3$ :
$1)\  123 \sim_s 321 \sim_s 213 \quad 2)\ 231 \sim_s 312 \quad 3)\ 132$.
 Bloom and Elizalde \cite[Table 3]{matching_partition} extended this investigation to pairs of patterns of length $3$, and classified all shape-Wilf-equivalent sets of size two.
 Their result is as follows. 
 \begin{theorem}[Bloom and Elizalde]
 The  following seven shape-Wilf-equivalence classes exist among pairs of 
 patterns of length 3. Every pair that is not in this list belongs to Class 1. 
    $$1) \ \{123, 213\} \quad 2)\ \{123, 231\}\quad 3)\ \{ 123, 312\} $$ $$ 4) \ \{123, 321\} \quad 5)\ \{213, 321\} \quad 6)\ \{123,132\} \quad 7)\ \{ 132, 321\}$$ 
\end{theorem}

A natural progression of this line of research is the classification of 
shape-Wilf-equivalence classes for larger  collections of patterns of length $3$. 
In this article, we provide a complete classification for all triples of patterns 
of length $3$.  Denote a list of triples for brevity as in Table \ref{table6}.
With this notation, 
our main result is the following.

\begin{theorem}
\label{thm:11-classes}
With the list of triples denoted as in Table \ref{table6}, the triples in 
the second column of Table \ref{table1} are shape-Wilf-equivalent.  The third and
the last columns give the number of perfect matchingss of length $2n$ that avoid
these triples and the references.  Some sequences are already present 
in OEIS (with other interpretations) and are given in black colour.  
Some sequences arising from our results have been submitted as 
new entries in OEIS.  These are given in \blue{blue} colour.  
Recurrence for two results were fitted with AI assistance and these 
are given in \red{red} colour.  This PDF file is best seen on a colour monitor,

\begin{table}
    \centering
    \scriptsize{
{\renewcommand{\arraystretch}{1.5}
\begin{tabular}{|l|c|l|l|l|l|}
\noalign{\hrule height 1pt}
\textbf{Class }&\textbf{Shape-Wilf-equivalent triples $\mathbf{P}$} &
$\mathbf{|\mathcal{M}_n(P)|~(n=1,2,\ldots,8,\ldots)}$ & ~~~~~~\textbf{Ref and OEIS}\\
\noalign{\hrule height 1pt}

 \hyperref[tripI]{I} &	
 $P_5 \sim_s P_6 \sim_s P_{11} \sim_s$  &\blue{$1,3,12,53,244,1146,5440,25981, \ldots$} & \blue{Theorem \ref{class1}, \href{https://oeis.org/A399989}{A399989}}\\ 
 & ~~~~~~~$P_{12}\sim_s P_{15}\sim_s P_{16}$ & & \\

\noalign{\hrule height 1pt}

\hyperref[tripII]{II}  & 
$P_8\sim_sP_{14} \sim_s P_{17}$ &  $1, 3, 12, 54, 259, 1294, 6655, 34986,\ldots$ & Theorem \ref{class2}, ~~\href{https://oeis.org/search?q=A125188&language=english&go=Search}{A125188}\\

\noalign{\hrule height 1pt}

\hyperref[tripIII]{III} &  
$P_3 \sim_sP_{18}$ & $1, 3, 12, 55, 273, 1428, 7752, 43263,\ldots$ &  Theorem \ref{thm: FC2}, ~~\href{https://oeis.org/search?q=A001764&language=english&go=Search}{A001764} \\

\noalign{\hrule height 1pt}
\hyperref[tripIV]{IV} & 
$P_2\sim_s P_{19}$  & $ 1, 3, 12, 55, 273, 1428, 7752, 43263,\ldots$ & Theorem \ref{conj(F)} \& \ref{thm: FC2}\\

\noalign{\hrule height 1pt}
V & 
$P_1$ & \red{$1,3,12,55,271,1400,7471,40841, \ldots$} & \red{ Theorem \ref{thm:matching-fibonacci-catalan} }\\
\noalign{\hrule height 1pt}
VI & 
$P_4$ & $\blue{1,3,12,51,217,925,3942, 16801,\ldots}$  & \blue{Theorem \ref{linear_triple4}, \href{https://oeis.org/A399164}{A399164}} \\

\noalign{\hrule height 1pt}
VII & 
$P_7$ & $1, 3, 12, 51, 219, 942, 4053, 17439,\ldots$& Theorem  \ref{linear_triple1},~~\href{https://oeis.org/search?q=A081704&language=english&go=Search}{A081704}\\
\noalign{\hrule height 1pt}
VIII& 
$P_9$ & \blue{$1,3,12,50,210,884,3724, 15692,\ldots$ }& \blue{Theorem \ref{linear_triple2}, \href{https://oeis.org/A400996}{A400996}}\\

\noalign{\hrule height 1pt}
IX & 
$P_{10}$ & $1,3,12,50,210,884,3724,15692,\ldots$ & Theorem \ref{conj(F)}, ~\ref{linear_triple2}\\
\noalign{\hrule height 1pt}
X & 
$P_{13}$ & \red{$1,3,12,54,258,1276,6449,33067,\ldots$} & \red{ Theorem \ref{thm:matching-initial-nesting} } \\
\noalign{\hrule height 1pt}
XI & 
$P_{20}$ &  $1, 3, 12, 55, 273, 1428, 7752, 43263,\ldots$ &  Theorem \ref{thm: FC1}, ~~\href{https://oeis.org/search?q=A001764&language=english&go=Search}{A001764}\\
\noalign{\hrule height 1pt}

\end{tabular}
}}
\caption{The shape-Wilf-equivalent classes for triples of patterns of length $3$}
\label{table1}
\end{table}
\end{theorem}


The proof of Theorem \ref{thm:11-classes}
appears in Section \ref{sec:swe-results}.
When we first uploaded on ArXiV, recurrence relations for 
$P_1$ and $P_{13}$ were still open.  Subsequently, we used GPT6 to fit a recurrence.   We have however proved both
recurrence relations by ourselves.  These 
are Theorem \ref{thm:matching-fibonacci-catalan}  and 
Theorem \ref{thm:matching-initial-nesting} and are 
presented in Section \ref{sec:ai_assist}.
In a subsequent paper (see \cite[Tables 3, 5]{suk-part-2-quad-quint}) we 
completely classify shape-Wilf equivalence of quadruple and quintuple patterns of length $3$. 

Bloom and Elizalde in \cite[Section 2.2]{matching_partition} gave a 
bijection between an $n\times n$ ``Ferrers board along 
with a transversal" and a perfect matching on $[2n]$.  Using 
this connection, we enumerate the number of perfect matchings avoiding triples of patterns of length three.
This is shown in the fourth column of Table \ref{table1}.
All our results enumerating perfect matchings appear in Section \ref{sec:perfect_matchings}.
Our results give new objects (class III and class XI) counted by the 
Fuss-Catalan numbers (\hyperlink{https://oeis.org/search?q=A001764}{A001764}) 
and two other sequences (corresponding to classes II and VII) from OEIS 
\hyperlink{https://oeis.org/search?q=A081704}{A081704}, \hyperlink{https://oeis.org/search?q=A125188}{A125188}.

\section{Encoding Framework and proof strategy}
\label{sec:encoding-framework}

To establish shape-Wilf-equivalences, we use the encoding method introduced in \cite{SUK-arxiv} and in \cite{SB-arxiv}. Since this method is used repeatedly throughout the paper, 
we describe the general framework here.  We will need a separate lemma for each equivalence class.
Let $P$ and $Q$ be two pattern sets for which we wish to prove $P\sim_s Q.$
Fix a Ferrers board $F$. Rather than constructing a bijection directly between the sets
$\SSS_F(P)$ and $\SSS_F(Q),$
we bijectively associate an encoding set of words $\WW_F$ to both $\SSS_F(P)$ and to $\SSS_F(Q)$.
Since $F$ is arbitrary, this proves $P\sim_s Q.$

The encoding is constructed inductively by considering the entries of a transversal 
row by row from top to bottom. At each stage, the rows and columns that have already 
been used are removed, leaving a smaller Ferrers board consisting of the remaining 
cells. 

\vspace{1 mm}

\noindent
{\bf Admissible and forced positions :}
Let $P$ be a pattern set.  Start with all cells of the Ferrers board coloured white and suppose that some initial entries of a transversal have already been placed according to the procedure described below. A white cell in the current top row is called \emph{$P$-admissible} if placing a $1$ in that cell admits a completion to a $P$-avoiding transversal of the remaining board.
If the current top row has exactly one $P$-admissible position, then the placement 
at that stage is said to be \emph{forced}. Otherwise, the placement is said to be 
\emph{unforced}.

The admissible positions at each row determines the possible letters in the encoding word. For each class of triples, we prove structural lemmas describing these positions 
and the geometric conditions for placement.

\vspace{1 mm}

\noindent
{\bf Inductive placement process :}
Let $F$ be a Ferrers board and let $T$ be a transversal of $F$ (recall these cells have a $1$). 
To construct the encoding of $T$, remove all the $1$'s from $F$ and place them back row by row from top to bottom 
in their original positions.
We follow the procedure below.

\begin{enumerate}[leftmargin=1.5cm]
    \item[Step 1:] Start with the Ferrers board $F$ with all cells white.

    \item[Step 2:] Place the $1$ belonging to the top row in its original position in $T$. Color this row and the column containing the $1$ gray.

    \item[Step 3:] Suppose that the first $i-1$ entries have already been placed. Consider the topmost row of the remaining white cells and place the next $1$ in its original position in $T$. Then color the corresponding row and column gray.
    \item[Step 4:] Continue this process until all the $1$'s of $T$ have been placed.
\end{enumerate}

Note that all cells of $F$ will be gray at the end of this process.  After each step, the remaining white cells form the board 
for the next iteration. 
We refer to this algorithm as the \emph{inductive placement process}.

\paragraph{General encoding scheme.}
Let $T\in\SSS_F(P)$. Apply the inductive placement process to $T$. At the $i$-th stage, suppose that the $P$-admissible cells in the current top row, listed from left to right, are
$a_{i,1}<a_{i,2}<\cdots<a_{i,r_i}.$

For each equivalence class considered below, we specify a finite alphabet and assign a symbol to each possible admissible choice. The $i$-th letter $w_i$ records which admissible position is occupied by the $1$ of $T$ at the $i$-th stage.

In particular, when the placement is forced, that is, when $r_i=1$, we use the symbol designated for a forced placement in the corresponding encoding. When more than one admissible position is available, the remaining symbols distinguish the possible choices according to the rule specified for that class
Proceeding through all rows produces a word
$
w=w_1w_2\cdots w_n,
$
called the \emph{encoding word} of $T$.

The inductive placement process itself is the same throughout this paper. What varies from one class to another is only the set of admissible positions and the symbols assigned to the possible choices.  The following remark is straightforward and will be used.

\begin{remark}\label{lem: atleast1} 
    At any stage of the process, there is always at least one position in the topmost row 
    where a $1$ can be placed without creating an occurrence of $P$. 

    At any stage of the encoding process, if we obtain no position to place $1$, then no transversal is valid for the corresponding Ferrers board $F$.
\end{remark}

\paragraph{Decoding process.}
Conversely, suppose that $w=w_1w_2\cdots w_n$ is an admissible encoding word for the pattern set under consideration. Start with the Ferrers board $F$ with all cells white and read the letters of $w$ from left to right.

At the $i$-th stage, determine the admissible positions in the current top row. The letter $w_i$, together with the encoding rule for the given pattern set, specifies the position in which the next $1$ is placed. If the step is forced, the unique admissible position is chosen. The corresponding row and column are then colored gray, and the procedure is repeated on the remaining white cells.
In this way, an admissible word determines the transversal recursively. Thus, once it is shown that every admissible word can be decoded uniquely and that the resulting transversal avoids the required pattern set, the encoding map is bijective.

\paragraph{Proving shape-Wilf-equivalence:}
For each family of pattern sets considered in the following sections, the proof therefore has three main ingredients.

First, we determine the admissible positions in the current top row. In particular, we identify the situations in which the next placement is forced.

Second, we show that the encoding rules are compatible across the pattern sets under consideration. More precisely, the structural lemmas imply that the same words are admissible for each pattern set, although the actual cells represented by a given letter may be different.

Finally, we decode a common word recursively for each pattern set. At every stage, after placing a $1$ and deleting its row and column, the remaining white cells form a smaller instance of the same problem. This allows the construction to proceed inductively until the entire transversal has been recovered.


Hence, we have $
|\SSS_F(P)|=|\SSS_F(Q)|
$
for every Ferrers board $F$ and thus $P\sim_s Q.$
We use this framework throughout the paper unless stated otherwise. To avoid repetition, in each equivalence class we give only the structural lemmas needed to determine the admissible and forced positions and to verify that the corresponding encoding and decoding procedures are compatible.


\section{
Shape Wilf Equivalent Classes for triples of Length-\texorpdfstring{$3$}{3}} 
\label{sec:swe-results}

We label the 
$\binom{6}{3}$ sets of permutation patterns as in Table \ref{table6}.

\begin{table}[H]
\centering
\begin{tabular}{c|c|c|c|c|c}
Label & Pattern set & Label & Pattern set & Label & Pattern set \\
\hline
$P_1$ & $\{123,132,213\}$ & $P_8$ & $\{123, 231, 312\}$ & $P_{15}$ & $\{132, 231, 321\}$\\
$P_2$ & $\{123,132,231\}$ & $P_9$ & $\{123, 231, 321\}$ & $P_{16}$ & $\{132, 312, 321\}$\\
$P_3$ & $\{123,132,312\}$ & $P_{10}$ & $\{123,312,321\}$ & $P_{17}$ & $\{213, 231, 312\}$\\

$P_4$ & $\{123,132,321\}$ & $P_{11}$ & $\{132, 213, 231 \}$ & $P_{18}$ & $\{213, 231, 321\}$\\

$P_5$ & $\{123, 213, 231\}$ & $P_{12}$ & $\{132, 213, 312\}$ & $P_{19}$ & $\{213, 312, 321\}$\\
$P_6$ & $\{123, 213, 312\}$ & $P_{13}$ & $\{132, 213,321\}$ & $P_{20}$ & $\{231, 312, 321\}$\\
$P_7$ & $\{123, 213,321\}$ & $P_{14}$ & $\{132, 231, 312 \}$ & &\\
\end{tabular}
\caption{Triple pattern sets considered in this paper}
\label{table6}
\end{table}

Before establishing the complete classification of the shape-Wilf-equivalence classes of triples, we first recall the following result \cite[Lemma 6, Propositions 15, 16]{sim-sch-rp} of Simion and Schmidt which gives the complete 
classification of their Wilf-equivalence classes.
\begin{theorem}[Schmidt and Simion]
When $n\ge 5$, we have the Wilf-equivalence classes.
\begin{itemize}
    \item The permutations in Classes I, II, III, IV, and X are Wilf-equivalent, and for any representative $P$ from these classes, $|\SSS_n(P)| = n.$
    \item The permutations in Classes VI, VII, VIII, and IX are Wilf-equivalent, and for any representative $P$ from these classes, $|\SSS_n(P)| = 0.$
    \item The permutations in Classes V and XI are Wilf-equivalent, and for any representative $P$ from these classes, $|\SSS_n(P)| = F_{n+1}.$
\end{itemize}
\end{theorem}

Conjugation of Ferrers boards provides a useful symmetry for relating pattern-avoiding transversals. In particular, conjugating the board corresponds to replacing a pattern by its inverse, as stated in the following theorem.

\begin{theorem}\label{conj(F)}
If a transversal in a  Ferrers board $F$ avoids (respectively, contains) a pattern $p$, then the transversal in conjugate of $F$, denoted by $\mathsf{Conj}(F)$, avoids (respectively, contains) $p^{-1}$. 
\end{theorem}

\begin{proof}
We prove the statement for containment; the argument for avoidance is analogous. 
Let $F$ contain the length $k$ pattern $p$. 
Then, there exists a transversal $T$ of $F$ and columns $i_1,i_2,\ldots,i_k$ such that the cells $(\pi_j, i_j)$ contain $1$ for $1 \leq j \leq k$ in $T$, forming an occurrence of $p$.
Now consider the conjugate board $\mathsf{Conj}(F)$. In the transpose of $T$, the cells $(i_j, \pi_j)$ contain $1$ for $1 \leq j \leq k$. Hence, this configuration forms the pattern $p^{-1}$ in $\mathsf{Conj}(F)$, completing the proof. 
\end{proof}

\subsection{Class I}\label{tripI}
In this section, we classify 
class I listed in Table \ref{table1}. It consists of the sets $P_5, P_6, P_{11},P_{12},P_{15},P_{16}$. The following shape-Wilf-equivalence was proved in \cite{SUK-arxiv}.

\begin{theorem}[{\cite[Theorem $6$]{SUK-arxiv}}] The sets of patterns  
 $P_{6}=\{123,213,312\}, P_{15}=\{132,231,321\}$ are shape-Wilf-equivalent.  
\end{theorem}

\subsubsection{Equivalence of \texorpdfstring{$P_5,P_6,P_{11},P_{16}$}{P5,P6,P11,P16}}

\begin{lemma}\label{lem: atmost2}
    At any stage of the process, the topmost row of white cells contains at most two positions where a $1$ can be placed while still avoiding the pattern set.
    \begin{itemize}
        \item For $P_5=\{123,213,231\}$,  the only valid positions are the first two cells in that row.
        \item For $P_{16}=\{132,312,321\}$, the only valid positions are the last two cells in that row.
        \item For $P_{11}=\{132,213,231\}$, the only valid positions are the first and last cells in that row.
    \end{itemize}
\end{lemma}
\begin{proof}
    At the current stage, if the topmost row has length greater than $3$, then to avoid the 
    pattern set we cannot place the $1$ outside the designated valid positions. Placing it 
    elsewhere would force the creation of a forbidden pattern.  Thus, we have the following.
    \begin{itemize}
        \item For $P_5=\{123,213,231\}$,  the $1$'s that will later be placed in the first two columns, together with this misplaced $1$, would form an occurrence of either 
        $123$ or $213$ and hence an element of $P_5$.

        \item For $P_{16}=\{132,312,321\}$,   the $1$'s that will later be placed in the last two columns, together with this misplaced $1$, would form an occurrence of either $312$ or $321$, and thus a pattern from $P_{16}$.

        \item For $P_{11}=\{132,213,231\}$,   the $1$'s that will later be placed in the first and last columns, together with this misplaced $1$, would form an occurrence of  either $132$ or $231$, and hence an element of $P_{11}$.
    \end{itemize}
    The proof is complete.
\end{proof}


\begin{lemma}\label{lem: rect}
     Let $F$ be a Ferrers board. At some stage $i$ of the 
     insertion processes for the pattern set, let $F'$ denote the board consisting of the remaining white cells of $F$. Suppose that the top row of $F'$ corresponds to row $r$ and the columns $c_1<\dots<c_l$ of $F$. Then in each of the following situations there is at most one suitable position for placing the next $1$.
    \begin{itemize}
        \item[\textnormal{(1)}] \textbf{Process for $P_5=\{123,213,231\}$:} If a $1$ from stage $i-1$ lies in the $(r+1)$th row between columns $c_1$ and $c_2$, then the only suitable position for the next $1$ is the right one, namely column $c_2$.

        \item[\textnormal{(2)}] \textbf{Process for $P_{16}=\{132,312,321\}$:} If a $1$ from stage $i-1$ lies in the $(r+1)$-th row between columns $c_{l-1}$ and $c_l$, then the only suitable position for the next $1$ is the left one, namely column $c_{l-1}$.

        \item[\textnormal{(3)}] \textbf{Process for $P_{11}=\{132,213,231\}$:} If the cells of $F$ lying above and to the right of $(c_l,r)$ already contain a $1$, then to avoid $P$ the next $1$ can be placed in at most one position, namely the right column $c_l$.

    \end{itemize}
\end{lemma}
\begin{proof}
    We show that in each situation described in the lemma, placing a $1$ in any other 
    position 
    creates a forbidden pattern.
    \begin{itemize}
        \item[\textnormal{(1)}] $\mathbf{ P_5 :}$ Suppose that there is a $1$ in column $c$ from stage $i-1$, lying between columns $c_1$ and $c_2$ in row $r+1$. If at stage $i$ we place the next $1$ in column $c_1$, then at some later stage a $1$ must be placed in column $c_2$ below row $r$. The three columns $c_1$, $c$, and $c_2$ would then form the pattern $231$, giving a contradiction.

        \item[\textnormal{(2)}] $\mathbf{ P_{16} :}$ Suppose that there is a $1$ in column $c$ from stage $i-1$, lying between columns $c_{l-1}$ and $c_l$ in row $r+1$. If at stage $i$ we place the next $1$ in column $c_l$, then at some later stage a $1$ must be placed in column $c_{l-1}$ below row $r$. The three columns $c_{l-1}$, $c$, and $c_l$ would then form the pattern $132$, giving a contradiction.

        \item[\textnormal{(3)}] $\mathbf{ P_{11} :}$ Suppose that there is a $1$ already present above and to the right of the cell $(c_l,r)$, lying in column $c$ from stage $i-1$. If at stage $i$ we place the next $1$ in column $c_1$, then at some later stage a $1$ must be placed in column $c_l$ below row $r$. The three columns $c_1$, $c_l$, and $c$ would then form the pattern $213$, giving a contradiction.

    \end{itemize}
    This completes the proof.
\end{proof}

\paragraph{Encoding Process:}\label{lem: encoding}
Let $T$ be a transversal of $F$ that avoids pattern set, and let 
$w = w_1 \dots w_n$ be its encoding over the alphabet 
$\{0,1,2\}$. For the $1$ in the $i$-th row (from top), the
letter $w_i$ is assigned as follows, depending on the pattern set:

\vspace{2 mm}

\small{
\begin{tabular}{|l|l|} \hline
Pattern Set & If the placement in row $i$ is not forced, then set  \\ \hline
\textbf{$P_5=\{123,213,231\}$:} & $w_i = 
\begin{cases}
2, & \text{if the $1$ lies in the leftmost available column},\\
1, & \text{if it lies in the second column from the left}.
\end{cases}$ \\ \hline 
\textbf{$P_{16}=\{132,312,321\}$:}   &
$w_i =
\begin{cases}
1, & \text{if the $1$ lies in the second-last available column},\\
2, & \text{if it lies in any other non-forced column except the last.}
\end{cases}$ \\ \hline
\textbf{$P_{11}=\{132,213,231\}$:}   & 
$w_i =
\begin{cases}
2, & \text{if the $1$ lies in the leftmost available column},\\
1, & \text{if it lies in the rightmost available column}.
\end{cases}$ \\ \hline
\end{tabular}
}
\vspace{2 mm}

In all three cases, if the placement in row $i$ is forced, then set $w_i = 0$.

\begin{lemma}\label{lem: forced-descent} 
At any stage $i$, if $w_i=1$ and the topmost row of white cells has length greater than $2$, then all subsequent placements are forced until a stage 
where the top row has exactly two white cells. The next stage after that will no longer forced, unless the top row consists of a single white cell.
\end{lemma}
\begin{proof}
    Let $F'$ denote the Ferrers board of white cells at stage $i$. In each of the following cases, assume that the top row of $F'$ contains at least three white cells; then all rows below it also contain at least three white cells. We show that under this condition every placement is forced until the top row reduces to two white cells.
    \begin{itemize}
        \item[\textnormal{(1)}] $\mathbf{P_5=\{123,213,231\}}$:  
    If at stage $i$ the $1$ is placed in the second column of the top row, then at stage $i+1$ it must be placed in the third column of the next row to avoid a future placement in that third column giving the pattern $231$. Iterating this argument shows that each placement is forced until the top row has exactly two white cells; let $j$ be the first such stage.

    At stage $j-1$, each of the top two rows of the 
    resulting board $\hat{F'}$ contains three cells. The placements at stages $j-1$ and $j$ must therefore lie in the second and third columns, respectively. If the third row of $\hat{F'}$ also has three cells, then stage $j+1$ has a unique available cell; otherwise, any placement in a column where previous rows contain no white cells introduces no forbidden pattern.

    \item[\textnormal{(2)}] $\mathbf{P_{16}=\{132,312,321\}}$:
    The argument is similar. A placement in the second last column of the top row forces the next placement into the third last column to avoid the pattern $132$. Thus all placements are forced until the top row has exactly two white cells; denote this stage by $j$.

    At stage $j-1$, the top two rows of $\hat{F'}$ have three cells each. The placements at stages $j-1$ and $j$ are forced into the second and first columns, respectively. If the third row of $\hat{F'}$ has exactly three cells, then stage $j+1$ has a single available cell; otherwise, placing the $1$ in a column unused by higher rows avoids all forbidden patterns.

    \item[\textnormal{(3)}] $\mathbf{P_{11}=\{132,213,231\}}$:
    Placing a $1$ in the last column of the top row forces placement in the second row to avoid the first column (else a forbidden pattern arises). Hence the placement must be in the second last available cell. Repeating this shows that subsequent placements are forced until the top row contains exactly two white cells; let $j$ be that stage.

    At stage $j-1$, the top two rows of $\hat{F'}$ contain exactly three white cells in columns $1,2,3$. The placements at stages $j-1$ and $j$ are forced into columns $3$ and $2$, respectively. If the third row of $\hat{F'}$ has exactly three cells, the placement at stage $j+1$ is uniquely determined; otherwise, any placement in columns unused by previous rows introduces no forbidden pattern.
 \end{itemize}
 This completes the proof.
\end{proof}

\begin{lemma}\label{lem: unforced-descent} 
    If $w_i=2$, the next stage is non-forced  except when the top row has exactly one white cell.
\end{lemma}
\begin{proof}
   We prove the statement for $P_5$; the proofs for the other pattern sets are analogous. Let $F'$ be the board of white cells at step $i$. The top row of $F'$ contains at least two white cells. If both the first and second rows contain exactly two white cells, then after the placement at step $i$, the top row at step $i+1$ contains only one white cell, and hence the next placement is forced.
Now suppose the top row contains at least three white cells, say in columns $c_1<c_2<\cdots<c_l.$
If $w_i=2$, then the $1$ is placed in column $c_1$. Consequently, at step $i+1$, both columns $c_2$ and $c_3$ are admissible positions for the next $1$. Therefore, the $(i+1)$-th step is not forced.
\end{proof}

\begin{remark}
    In \cite[Lemmas 15, 16, 18, 19, 20]{SUK-arxiv},  we describe the encoding process for the pattern set $P_{6}=\{123,213,312\}$. Comparing those with Lemmas  \ref{lem: atmost2}, \ref{lem: rect}, \ref{lem: unforced-descent}, \ref{lem: forced-descent} and the encoding process, we  get the shape-Wilf-equivalence of $P_6$  to the above three pattern sets. Therefore $P_5\sim_s P_6 \sim_s P_{11} \sim_s P_{16}.$
\end{remark}

\subsubsection{Equivalence of $P_6,P_{12}$}
  For $P_{12}=\{132,213,312\}$, there are clearly no fixed two cells in which a $1$ entry must be placed.  

\begin{lemma}\label{lem: rect2}
Let $F$ be a Ferrers board, and let $\ell(c)$ denote the height of column $c$. At some stage of the construction for $P_{12}$, let $F'$ be the subboard of the remaining white cells. Suppose that the top row of $F'$ corresponds to the row $r$ of $F$ and occupies columns $c_1<\dots<c_l$. Then the next placement is uniquely determined in any of the following situations:

\begin{itemize}[leftmargin=1.5cm]
    \item[Case 1:] If a $1$ from stage $i-1$ lies in row $r+1$ between columns $c_k$ and $c_{k+1}$; then the only valid position is $c_k$.
    
    \item[Case 2:] If a $1$ occurs strictly above and to the right of $(c_l,r)$; then the only valid position is $c_l$.
    
    \item[Case 3:] If the rectangle with vertices $(1,r+1)$, $(c_1-1,r+1)$, $(c_1-1,\ell(c_2))$, $(1,\ell(c_2))$ contains a $1$; then the only valid position is $c_1$.
\end{itemize}
\end{lemma}
\begin{proof}
Each case follows from the fact that placing a $1$ at column $c$ outside the indicated forced position creates a forbidden pattern. 

\begin{itemize}[leftmargin=1.5cm]
    \item[Case 1:] If a $1$ from stage $i-1$ lies in row $r+1$ between columns $c_k$ and $c_{k+1}$, then placing the next $1$ in column $c_{k+1}$ forces a later $1$ in column $c_k$ below row $r$, producing the pattern $132$ from the columns $c_{k}$, $c$, and $c_{k+1}$, a contradiction. Hence only $c_k$ is valid.

    \item[Case 2:] If a $1$ occurs above and to the right of $(c_l,r)$, then placing the next $1$ in column $c_1$ forces a later $1$ in column $c_l$ below row $r$, yielding the pattern $213$ 
    Thus, only $c_l$ is valid.

    \item[Case 3:] If a $1$ lies in the specified rectangle, then placing the next $1$ in any column $c_j$ with $j \neq 1$ forces a later $1$ in column $c_1$ below row $r$, creating the pattern $312$ from the columns $c$, $c_1$, and $c_j$. Hence,
    only $c_1$ is valid.
\end{itemize}
The proof is complete.
\end{proof}

\paragraph{Encoding Process:} We define an encoding for transversals that avoid $P_{12}$.
Let $T$ be such a transversal of $F$, and let 
$w = w_1 \dots w_n$ be the associated word over $\{0,1,2\}$.  
For the $1$ in the $i$-th row (from the top), set:
\begin{enumerate}
    \item $w_i = 1$ if the $1$ lies in the leftmost column of the top row of $F'$ and its placement is not forced by the third case of Lemma \ref{lem: rect2};
    \item $w_i = 2$ if the placement is not in the leftmost column and is not forced by the third case of Lemma \ref{lem: rect2};
    \item $w_i = 0$ if the placement is forced by the third case of Lemma \ref{lem: rect2}, or if the top row has only one white cell.
\end{enumerate}

\begin{example}
    We work out an example for the board $(5,5,4,3,3)$. We can see that the transversal $4,5,3,2,1$ avoids $P_{12}$. We obtain the encoding word from the board and transversal. 
    \vspace{0.2cm}\\
\scalebox{0.75}{
\ytableausetup{centertableaux} 
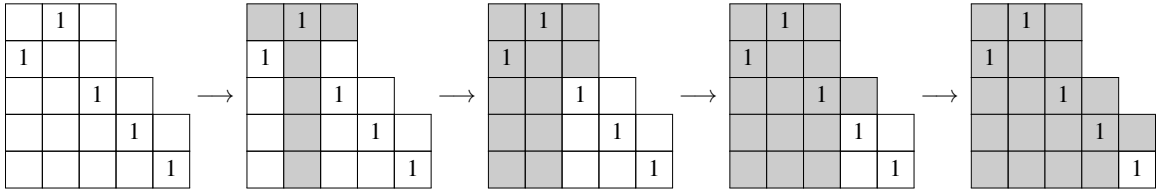
\begin{ytableau}
   \textcolor{white}{d}  & $1$ & \\
     $1$ & \textcolor{white}{d} &  \\
    \textcolor{white}{d} & & $1$ & \textcolor{white}{d}  \\
     & \textcolor{white}{d} & \textcolor{white}{d} & $1$& \\
     & \textcolor{white}{d} & \textcolor{white}{d} & & $1$ \\
\end{ytableau}
$\longrightarrow$
\ytableausetup{centertableaux} 
\begin{ytableau}
   *(lgrey)  & *(lgrey)$1$ & *(lgrey)\\
     $1$ & *(lgrey) & \\
    \textcolor{white}{d} & *(lgrey) & $1$ & \textcolor{white}{d} \\
     & *(lgrey) & \textcolor{white}{d} & $1$ &\\
     & *(lgrey) & \textcolor{white}{d} & & $1$  \\
\end{ytableau}
$\longrightarrow$
\ytableausetup{centertableaux} 
\begin{ytableau}
   *(lgrey)  & *(lgrey)$1$ & *(lgrey)\\
     *(lgrey)$1$ & *(lgrey) & *(lgrey) \\
    *(lgrey) & *(lgrey) & $1$ & \textcolor{white}{d} \\
     *(lgrey)& *(lgrey) &\textcolor{white}{d} & $1$ &\\
     *(lgrey)& *(lgrey) & \textcolor{white}{d} & & $1$ \\
\end{ytableau}
$\longrightarrow$
\ytableausetup{centertableaux} 
\begin{ytableau}
   *(lgrey)  & *(lgrey)$1$ & *(lgrey)\\
     *(lgrey)$1$ & *(lgrey) & *(lgrey) \\
    *(lgrey) & *(lgrey) & *(lgrey)$1$ & *(lgrey) \\
     *(lgrey)& *(lgrey) & *(lgrey) &  $1$ &\\
     *(lgrey)& *(lgrey) & *(lgrey) & & $1$  \\
\end{ytableau}
$\longrightarrow$
\ytableausetup{centertableaux} 
\begin{ytableau}
   *(lgrey)  & *(lgrey)$1$ & *(lgrey)\\
     *(lgrey)$1$ & *(lgrey) & *(lgrey) \\
    *(lgrey) & *(lgrey) & *(lgrey)$1$ & *(lgrey)  \\
     *(lgrey)& *(lgrey) & *(lgrey) &  *(lgrey)$1$ & *(lgrey)\\
     *(lgrey)& *(lgrey) & *(lgrey)  & *(lgrey) &$1$\\
\end{ytableau}
}\vspace{0.2cm}

Observe that $w_1=2$. The position in the second row (from top) is forced by first of Lemma \ref{lem: rect2} and lies in the leftmost available column; therefore, $w_2=1$. The placement in the third row (from top) is determined by third of Lemma \ref{lem: rect2}, so $w_3=0$. Similarly, $w_4=1$ since the corresponding entry is placed in the leftmost available column, and finally $w_5=0$. Hence, the encoding word is
$w=(2,1,0,1,0).$
\end{example}

As the proof of the next Lemma is very similar to that of Lemma
\ref{lem: forced-descent}, we merely state our result and omit its proof.

\begin{lemma}\label{lem: forced-descent2}
     For the pattern set $P_{12}$, at any stage $i$, if $w_i=1$ and the topmost row of white cells has length greater than $2$, then all subsequent placements are forced by the third case of Lemma \ref{lem: rect2} until a stage is reached where the top row has exactly two white cells. The next stage after that is no longer forced, unless the top row consists of a single white cell.
\end{lemma}


\begin{lemma}\label{lem: unforced-descent2}
    For the pattern set $P_{12}$, if $w_i=2$, then for the next $(i+1)$-th  stage, $w_{i+1}\neq 0$  except when the top row has exactly one white cell.
\end{lemma}
\begin{proof}
    Let $F'$ be the board of white cells at step $i$. If the topmost row of $F'$ 
    contains exactly two white cells, then the next placement is forced.
Now, suppose that the top row contains at least three white cells, say in columns
$c_1<c_2<\cdots<c_l.$

If $w_i=2$, then the $1$ is placed in column $c_j$ for some $j\neq 1$, and 
by the third case of Lemma \ref{lem: rect2}, this placement is not forced. Consequently, 
at step $i+1$, the placement can only be forced by the first or second case of Lemma \ref{lem: rect2}. Therefore, $w_{i+1}\neq 0.$
\end{proof}

\begin{remark}
      In \cite[Lemmas 15, 16, 18, 19, 20]{SUK-arxiv}, we describe the encoding process for the set of patterns $P_{6}=\{123,213,312\}$. Comparing them with Lemmas \ref{lem: rect2}, \ref{lem: forced-descent2}, \ref{lem: unforced-descent2} and the encoding process,  we  get the shape-Wilf-equivalence of $P_6$ and $P_{12}$, i.e., $P_6 \sim_s P_{12}$.
\end{remark}
Since the encoding process for $P_{12}$ is different and more intricate than the previous encodings, we describe how to reconstruct a transversal of $F$ avoiding $P_{12}$ from the encoding word $w = w_1 \cdots w_n$.  Towards this, we establish the following lemma.
\begin{lemma}\label{w_i=1}
For the pattern set $P_{12}$,
let $w = w_1 \ldots w_n$ be an encoding word, and let $i$ be the smallest index such that $w_i = 1$. Thus, the first column contains a $1$ in the $(n-i+1)$-st row from the bottom, equivalently in the $i$-th row from the top. Let $l$ denote the length of the top row of $F$. Then the top row contains a $1$ in the $i$-th column (from the left) if $l \ge i$, and in the $l$-th (rightmost) column if $l < i$.
\end{lemma}

\begin{proof}
Assume that the topmost (or $n$-th) row contains a $1$ in column $r$, where $r \le l$. 
In the $(n-1)$-st row, a $1$ must then appear either in column $(l-1)$ or in the rightmost column.
If the $(n-1)$-st row has a $1$ in column $(l-1)$, then in the $(n-2)$-th row the $1$ lies either in column $(l-2)$ or in the rightmost column.  
If the $(n-1)$-th row instead contains a $1$ in the rightmost column (which is not column $(l-1)$), then the $(n-2)$-th row contains a $1$ either in column $(l-1)$ or in the rightmost column.
Proceeding in this manner, the first column will contain a $1$ in the $(n-r+1)$-th row if $r < l$, 
or in some $j$-th row with $j < l$ (from the top) if $r = l$.
Hence, if $l \ge i$, we must have $r = i$, and if $l < i$, then we must have $r = l$,
completing the proof.
\end{proof}

We next describe reconstruction of a transversal of $F$ avoiding $P_{12}$ from 
its encoding word $w = w_1 \dots w_n$.

\noindent
{\bf Decoding Process :} 
Let $i$ be the smallest index such that $w_i = 1$. Then, a $1$ is placed in the
first column of the $i$-th row, and hence $w_j = 2$ for all $j<i$.
Let $l$ denote the length of the topmost row of $F$. By Lemma \ref{w_i=1}, the 
position of the $1$ in the top row is given by
\[
\text{top row:}\quad
\begin{cases}
\text{$1$ in the $i$-th column (from the left)}, & \text{if } l \ge i,\\
\text{$1$ in the $l$-th column (rightmost)}, & \text{if } l < i .
\end{cases}
\]
By Lemma \ref{lem: rect2}, all placements in the rows between the top row and the
$i$-th row are then forced: 
each row $j<i$ must place its $1$ in the rightmost available white cell.  
Thus the top $i$ rows are completely determined.

Next, suppose that the block $w_{i+1},\dots,w_{i'}$ consists entirely of zeros.  
Then, by the third case of Lemma \ref{lem: rect2}, every placement from the 
$(i+1)$-th up to the $i'$-th row is forced, meaning that each row has a 
unique admissible white cell for the $1$.  
After row $i'$, we apply the same procedure recursively to the remaining 
subword and the remaining white cells of the board.

If there is no index $i$ with $w_i = 1$, then the encoding must be of the 
form $22\cdots 20$. In this case, each row's $1$ is placed in the rightmost 
available white cell, beginning from the top row and proceeding downward, and 
the transversal is uniquely determined.
\begin{example}
     Consider the Ferrers board $F=(6,6,5,5,4,4)$ and the encoding word $w=2,2,1,0,2,0$. Here, $i=3$ is the smallest index such that $w_i=1$. Hence, the first column contains $1$ in the fourth row (i.e., the third row from the top). Since the length of the topmost row is $l=4>3$, the top row (the sixth row) contains $1$ in the third column, and the fifth row contains $1$ in the second column which is forced by second case of Lemma \ref{lem: rect2}.

Now $w_4=0$ implies that the placement in the third row is fixed by third case of Lemma \ref{lem: rect2}, namely in the fourth column. We are now left with the fifth and sixth columns. Since $w_5=2$, the second row contains $1$ in the sixth column, and the first row contains $1$ in the fifth column. Hence, the resulting transversal is $4,5,6,3,1,2$, avoiding $P_{12}$.

\begin{center}
\scalebox{0.75}{
\ytableausetup{centertableaux} \begin{ytableau}
    $0$& $0$ & $1$ & $0$\\
     $0$& $1$ & $0$ & $0$  \\
    $1$ & $0$ & $0$ & $0$ & $0$ \\
    $0$ & $0$ & $0$ & $1$ & $0$\\
    $0$ & $0$ & $0$ & $0$ & $0$ & $1$\\
    $0$ & $0$ & $0$ & $0$ & $1$ & $0$\\
\end{ytableau} }
\end{center}
\end{example}

\subsection{Class II}\label{tripII}
In this section, we show shape-Wilf-equivalence 
of $P_8,P_{14}$ and $P_{17}$.
\subsubsection{Equivalence of $P_8,P_{17}$}
In contrast to Case I, there are no distinguished cells in which the entry $1$ must 
be placed. The $1$ may be placed in any admissible cell of the row and the encoding 
is not based on a ternary alphabet $\{0,1,2\}$.

The following lemma identifies three situations in which the next placement in the 
encoding process is forced. Since each case follows by an argument analogous to the 
forced-position arguments established earlier, we omit the proof.
\begin{lemma}\label{forced}
Let $F$ be a Ferrers board, and let $\ell(c)$ denote the height of column $c$. At some stage of the construction, let $F'$ be the subboard of remaining white cells. Suppose the top row of $F'$ corresponds to row $r$ of $F$ and occupies columns $c_1<\dots<c_l$. Then the next placement is uniquely determined in any of the following  situations: 

\begin{itemize}[leftmargin=1.5cm]
    \item[Case 1:] A $1$ occurs strictly above and to the right of $(c_l,r)$; then the only valid position is $c_l$ for $P_{17}=\{213,231,312\}$ and  $c_1$ for $P_8=\{123,231,312\}$.
    \item[Case 2:] A $1$ from stage $i-1$ lies in row $r+1$ between columns $c_k$ and $c_{k+1}$; then the only valid position is $c_{k+1}$.
    \item[Case 3:] The rectangle with vertices $(1,r+1)$, $(c_1-1,r+1)$, $(c_1-1,\ell(c_2))$, $(1,\ell(c_2))$ contains a $1$; then the only valid position is $c_1$.
\end{itemize}
\end{lemma}

\noindent
{\bf Encoding Process :} We now define an encoding for transversals that avoid set of patterns. 
For  $P_8=\{123,231,312\}$, at each step of the encoding process, we reinterpret 
the column indices of the Ferrers board (restricted to the remaining white cells) 
by treating\emph{ the first column as the last column}, and relabeling each column 
$i \geq 2$ as column $i-1$. Therefore second column is now considering as the first column. 
For $P_{17}=\{213,231,312\}$, the column numbering is used in the standard way.
Let $T$ be such a transversal of $F$, and let 
$w = w_1 \dots w_n$ be the associated word over $\{0,1,\ldots,n\}$.  
For the $1$ in the $i$-th row (from the top), set:  
\begin{itemize}
   \item $w_i=j$ if it is in the $j$-th column of the subboard of remaining white cells and its placement is not forced. 
    \item $w_i = 0$ if the placement is forced.
\end{itemize}

\begin{example}
  We work out an example for the board $(5,5,5,3,3)$. We can see that the transversal $5,4,3,2,1$ avoids $P_{8}$. We obtain the encoding word from the board and transversal. 
    \vspace{0.2cm}\\
\scalebox{0.75}{
\ytableausetup{centertableaux} 
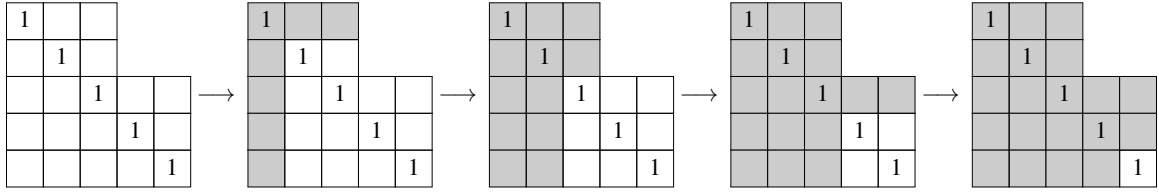
\begin{ytableau}
   $1$  &  & \\
     & $1$ &  \\
    \textcolor{white}{d} & & $1$ & \textcolor{white}{d} &  \\
     & \textcolor{white}{d} & \textcolor{white}{d} & $1$& \\
     & \textcolor{white}{d} & \textcolor{white}{d} & & $1$ \\
\end{ytableau}
$\longrightarrow$
\ytableausetup{centertableaux} 
\begin{ytableau}
   *(lgrey) $1$  & *(lgrey)& *(lgrey)\\
     *(lgrey)& $1$ & \\
    *(lgrey) &  & $1$ & \textcolor{white}{d} & \\
     *(lgrey) &  & \textcolor{white}{d} & $1$ &\\
    *(lgrey)  &  & \textcolor{white}{d} & & $1$  \\
\end{ytableau}
$\longrightarrow$
\ytableausetup{centertableaux} 
\begin{ytableau}
   *(lgrey)$1$  & *(lgrey) & *(lgrey)\\
     *(lgrey) & *(lgrey)$1$ & *(lgrey) \\
    *(lgrey) & *(lgrey) & $1$ & \textcolor{white}{d} &\\
     *(lgrey)& *(lgrey) &\textcolor{white}{d} & $1$ &\\
     *(lgrey)& *(lgrey) & \textcolor{white}{d} & & $1$ \\
\end{ytableau}
$\longrightarrow$
\ytableausetup{centertableaux} 
\begin{ytableau}
   *(lgrey) $1$ & *(lgrey) & *(lgrey)\\
     *(lgrey) & *(lgrey)$1$ & *(lgrey) \\
    *(lgrey) & *(lgrey) & *(lgrey)$1$ & *(lgrey) & *(lgrey) \\
     *(lgrey)& *(lgrey) & *(lgrey) &  $1$ &\\
     *(lgrey)& *(lgrey) & *(lgrey) & & $1$  \\
\end{ytableau}
$\longrightarrow$
\ytableausetup{centertableaux} 
\begin{ytableau}
   *(lgrey) $1$ & *(lgrey) & *(lgrey)\\
     *(lgrey) & *(lgrey)$1$ & *(lgrey) \\
    *(lgrey) & *(lgrey) & *(lgrey)$1$ & *(lgrey) & *(lgrey) \\
     *(lgrey)& *(lgrey) & *(lgrey) & *(lgrey) $1$ & *(lgrey)\\
     *(lgrey)& *(lgrey) & *(lgrey) &*(lgrey) & $1$  \\
\end{ytableau}
}\vspace{0.2cm}   

The topmost row has $1$ in the first column. As for $P_8$, we treat the first column  as last column, $w_1=3$. The next placement is forced by Lemma \ref{forced}, hence $w_2=0$. Now the next row contains two new columns and the $1$ is in the first position, hence $w_3=3$. Next two placements are forced therefore $w_4=w_5=0$. Hence, the
encoding word is $w = (3, 0, 3, 0, 0).$
The boards below display the insertion process for the transversal avoiding $P_{17}$.
\vspace{0.2cm}\\
\scalebox{0.75}{
\ytableausetup{centertableaux} 
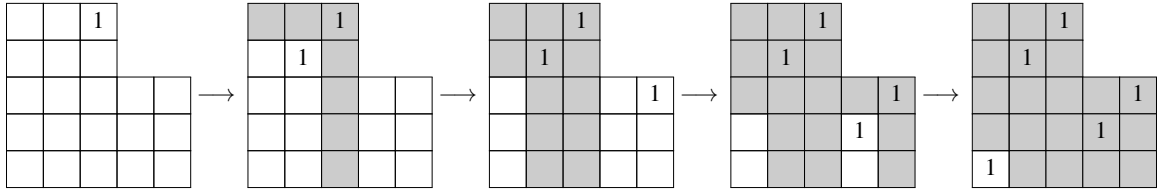
\begin{ytableau}
   \textcolor{white}{d}  &  & $1$ \\
     &   &  \\
    \textcolor{white}{d} & &  & \textcolor{white}{d} &   \\
     & \textcolor{white}{d} & \textcolor{white}{d} &  & \\
     & \textcolor{white}{d} & \textcolor{white}{d} & &   \\
\end{ytableau}
$\longrightarrow$
  \ytableausetup{centertableaux} 
\begin{ytableau}
   *(lgrey)  & *(lgrey)  & *(lgrey) $1$ \\
     & $1$  & *(lgrey) \\
    \textcolor{white}{d} & & *(lgrey) & \textcolor{white}{d} &   \\
     & \textcolor{white}{d} & *(lgrey) &  & \\
     & \textcolor{white}{d} & *(lgrey) & &   \\
\end{ytableau}

$\longrightarrow$
  \ytableausetup{centertableaux} 
\begin{ytableau}
   *(lgrey)  & *(lgrey)  & *(lgrey) $1$ \\
      *(lgrey)&  *(lgrey)$1$  & *(lgrey) \\
    \textcolor{white}{d} &  *(lgrey) & *(lgrey) & \textcolor{white}{d} &  $1$ \\
     &  *(lgrey) & *(lgrey) &  & \\
     &  *(lgrey) & *(lgrey) & &   \\
\end{ytableau}
$\longrightarrow$
  \ytableausetup{centertableaux} 
\begin{ytableau}
   *(lgrey)  & *(lgrey)  & *(lgrey) $1$ \\
      *(lgrey)&  *(lgrey)$1$  & *(lgrey) \\
    *(lgrey) &  *(lgrey) & *(lgrey) & *(lgrey) &  *(lgrey)$1$ \\
     &  *(lgrey) & *(lgrey) & $1$  & *(lgrey)\\
     &  *(lgrey) & *(lgrey) & & *(lgrey)  \\
\end{ytableau}
$\longrightarrow$
  \ytableausetup{centertableaux} 
\begin{ytableau}
   *(lgrey)  & *(lgrey)  & *(lgrey) $1$ \\
      *(lgrey)&  *(lgrey)$1$  & *(lgrey) \\
    *(lgrey) &  *(lgrey) & *(lgrey) & *(lgrey) &  *(lgrey)$1$ \\
     *(lgrey) &  *(lgrey) & *(lgrey) & *(lgrey)$1$  & *(lgrey)\\
     $1$ &  *(lgrey) & *(lgrey) & *(lgrey)& *(lgrey)  \\
\end{ytableau}

}\vspace{0.2cm}   

$w_1=3$ implies $1$ must be in the third position in topmost row. $w_2=0$ implies the next position is forced in the second column. Again $w_3=3$ implies $1$ must be in the third position in the third row. Next two placements are forced in fourth and first column respectively. Hence the transversal obtained is $1,4,5,2,3$ which avoids $P_{17}$.
\end{example}

\subsubsection{Equivalence of $P_8,P_{14}$}
\begin{lemma}\label{lem: atmost2_II}
    At any stage of the process, the topmost row of white cells admits at most two admissible positions for placing a $1$ while avoiding the pattern set $P_{14}=\{132,231,312\}$. These valid positions are precisely the first and the last cells of that row.
\end{lemma}
\begin{proof}
    At the current stage, if the topmost row has length greater than $3$, then to avoid $P$ the $1$ cannot be placed outside the designated admissible positions. Indeed, if the $1$ is placed in any column other than the first or the last, then the $1$'s that must later appear in the first and last columns will, together with this misplaced $1$, form an occurrence of either $132$ or $231$, completing the proof.
\end{proof}

For the pattern set $P_{14}$, the following lemma gives another configuration in which the next placement is forced. The proof is analogous to the corresponding rectangle argument used earlier, and hence we omit the details.
\begin{lemma}\label{lem: rect_II}
    Let $F$ be a Ferrers board, and let $\ell(c)$ denote the height of column $c$. At some stage of the construction for $P_{14}=\{132,231,312\}$, let $F'$ be the subboard of the remaining white cells. Suppose that the top row of $F'$ corresponds to the row $r$ of $F$ and occupies columns $c_1<\dots<c_l$. Then the next placement is uniquely determined if the rectangle with vertices $(1,r+1)$, $(c_1-1,r+1)$, $(c_1-1,\ell(c_2))$, $(1,\ell(c_2))$ contains a $1$. The only valid position is $c_1$.
\end{lemma}

\noindent
{\bf Encoding Process for $\mathbf{P_{14}}$:}
Let $T$ be such a transversal of $F$, and let 
$w = w_1 \dots w_n$ be the associated word over $\{0,1,2\}$.  
For the $1$ in the $i$-th row (from the top), set:  
\begin{itemize}
    \item $w_i=2$ if it is in the last column of the subboard of the remaining white cells and its placement is not forced. 
   \item $w_i=1$ if it is in the first column of the subboard of the remaining white cells and its placement is not forced. 
   
    \item $w_i = 0$ if the placement is forced.
\end{itemize}

\noindent
{\bf Encoding Process for $\mathbf{P_{8}}$:}
Let $T$ be such a transversal of $F$, and let 
$w = w_1 \dots w_n$ be the associated word over $\{0,1,2\}$.  
For the $1$ in the $i$-th row (from the top), set:  
\begin{itemize}
    \item $w_i=2$ if it is not in the first column of the subboard of the remaining white cells and not forced by the third case of Lemma \ref{forced}.
    \item $w_i=1$ if it is in the first column of the subboard of the remaining white cells and not forced by the third case of Lemma \ref{forced}.
    
    \item $w_i=0$ if the placement is forced by the third case of Lemma \ref{forced} or if the top row has only one white cell.
\end{itemize}
\begin{remark}
   Note that for both pattern sets $P_8$ and $P_{14}$, the conditions corresponding to $w_i=0$ and $w_i=1$ are identical. For $w_i=2$, however, the condition that the entry lies in the last column for $P_{14}$ is equivalent to the condition that it does not lie in the first column for $P_8$.
\end{remark}
\begin{example}
     We work out an example for the board $(5,5,5,4,3)$. We can see that the transversal $1,2,5,4,3$ avoids $P_{14}$. We obtain the encoding word from the board and transversal. 
    \vspace{0.2cm}\\
\scalebox{0.75}{
\ytableausetup{centertableaux} 
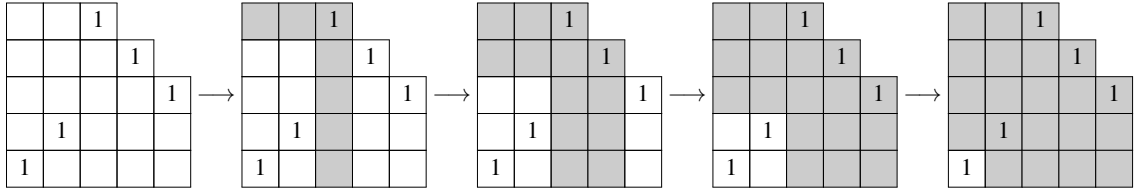
\begin{ytableau}
   \textcolor{white}{d}  &  & $1$\\
     &  &  & $1$\\
    \textcolor{white}{d} & & & \textcolor{white}{d} & $1$ \\
     & $1$ & \textcolor{white}{d} & & \\
     $1$& \textcolor{white}{d} & \textcolor{white}{d} & & \\
\end{ytableau}
$\longrightarrow$
\begin{ytableau}
   *(lgrey)  & *(lgrey) & *(lgrey)$1$\\
     &  & *(lgrey) & $1$\\
    \textcolor{white}{d} & & *(lgrey) &  & $1$ \\
     & $1$ & *(lgrey) & & \\
     $1$& \textcolor{white}{d} & *(lgrey) & & \\
\end{ytableau}
$\longrightarrow$
\begin{ytableau}
   *(lgrey)  & *(lgrey) & *(lgrey)$1$\\
    *(lgrey) & *(lgrey) & *(lgrey) & *(lgrey)$1$\\
    \textcolor{white}{d} & & *(lgrey) & *(lgrey) & $1$ \\
     & $1$ & *(lgrey) & *(lgrey)& \\
     $1$& \textcolor{white}{d} & *(lgrey) & *(lgrey) & \\
\end{ytableau}
$\longrightarrow$
\begin{ytableau}
   *(lgrey)  & *(lgrey) & *(lgrey)$1$\\
    *(lgrey) & *(lgrey) & *(lgrey) & *(lgrey)$1$\\
    *(lgrey) & *(lgrey) & *(lgrey) & *(lgrey) & *(lgrey)$1$ \\
     & $1$ & *(lgrey) & *(lgrey)& *(lgrey)\\
     $1$& \textcolor{white}{d} & *(lgrey) & *(lgrey) & *(lgrey)\\
\end{ytableau}
$\longrightarrow$
\begin{ytableau}
   *(lgrey)  & *(lgrey) & *(lgrey)$1$\\
    *(lgrey) & *(lgrey) & *(lgrey) & *(lgrey)$1$\\
    *(lgrey) & *(lgrey) & *(lgrey) & *(lgrey) & *(lgrey)$1$ \\
     *(lgrey)& *(lgrey)$1$ & *(lgrey) & *(lgrey)& *(lgrey)\\
     $1$& *(lgrey) & *(lgrey) & *(lgrey) & *(lgrey)\\
\end{ytableau}
}\vspace{0.2cm}   

The encoding word is $2,2,2,2,0$. The boards below display the insertion process for the transversal avoiding $P_8$.\vspace{0.2cm}\\
\scalebox{0.75}{
\ytableausetup{centertableaux} 
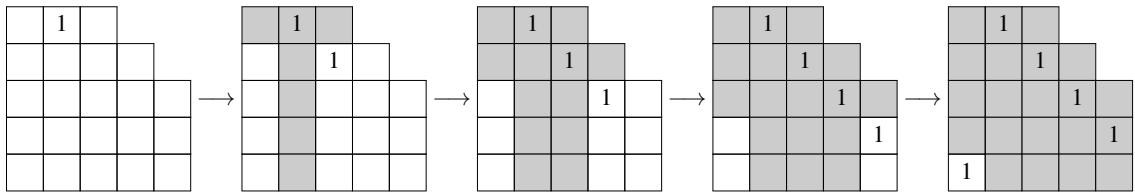
\begin{ytableau}
   \textcolor{white}{d}  & $1$ & \\
     &  & & \\
    \textcolor{white}{d} & & & \textcolor{white}{d} &  \\
     &  & \textcolor{white}{d} & & \\
     & \textcolor{white}{d} & \textcolor{white}{d} & & \\
\end{ytableau}
$\longrightarrow$
\begin{ytableau}
  *(lgrey)  & *(lgrey)$1$ & *(lgrey)\\
     & *(lgrey) & $1$& \\
    \textcolor{white}{d} & *(lgrey) & & \textcolor{white}{d} &  \\
     & *(lgrey) & \textcolor{white}{d} & & \\
     & *(lgrey) & \textcolor{white}{d} & & \\
\end{ytableau}
$\longrightarrow$

\begin{ytableau}
  *(lgrey)  & *(lgrey)$1$ & *(lgrey)\\
    *(lgrey)  & *(lgrey) & *(lgrey) $1$& *(lgrey) \\
    \textcolor{white}{d} & *(lgrey) & *(lgrey) & $1$ &  \\
     & *(lgrey) & *(lgrey)  & & \\
     & *(lgrey) & *(lgrey)  & & \\
\end{ytableau}
$\longrightarrow$

\begin{ytableau}
  *(lgrey)  & *(lgrey)$1$ & *(lgrey)\\
    *(lgrey)  & *(lgrey) & *(lgrey) $1$& *(lgrey) \\
    *(lgrey) & *(lgrey) & *(lgrey) & *(lgrey)$1$ & *(lgrey) \\
     & *(lgrey) & *(lgrey)  & *(lgrey)& $1$\\
     & *(lgrey) & *(lgrey)  &*(lgrey) & \\
\end{ytableau}
$\longrightarrow$
\begin{ytableau}
  *(lgrey)  & *(lgrey)$1$ & *(lgrey)\\
    *(lgrey)  & *(lgrey) & *(lgrey) $1$& *(lgrey) \\
    *(lgrey) & *(lgrey) & *(lgrey) & *(lgrey)$1$ & *(lgrey) \\
     *(lgrey)& *(lgrey) & *(lgrey)  & *(lgrey)& *(lgrey)$1$\\
     $1$& *(lgrey) & *(lgrey)  &*(lgrey) & *(lgrey)\\
\end{ytableau}
}\vspace{0.2cm} 

The condition $w_1=2$ implies that, in the topmost row, the entry $1$ can be placed either in the second or the third column. Suppose it is placed in the third column. Then, at the next stage, the placement of $1$ in the first column is forced by first case of Lemma \ref{forced}, which would imply $w_2=1$, a contradiction. Therefore, the entry $1$ in the topmost row must be placed in the second column.
In the next row, the entry $1$ cannot be placed in the rightmost column, as this would eventually create an occurrence of either the pattern $123$ or $231$. Hence, the second column is the only valid choice. By the same reasoning, the third column is the unique valid position for the entry in the third row.
The placement in the second row is then forced to the rightmost column by second case of Lemma \ref{forced}. Finally, only one remaining white cell in the first row. Consequently, the resulting transversal is $1,5,4,3,2$, which avoids $P_8$. 
\end{example}

\subsection{Class III}\label{tripIII}
In this section, we classify the shape-Wilf-equivalence class III listed in Table \ref{table1}. 
\begin{lemma}\label{lem: rect_III}
Let $F$ be a Ferrers board, and let $\ell(c)$ denote the height of column $c$. At some stage of the construction, let $F'$ be the subboard of the remaining white cells. Suppose that the top row of $F'$ corresponds to the row $r$ of $F$ and occupies columns $c_1<\dots<c_l$. Then the next placement is uniquely determined in any of the following situations:

\begin{itemize}[leftmargin=1.5cm]
    \item[Case 1:] A $1$ from stage $i-1$ lies in row $r+1$ between columns $c_k$ and $c_{k+1}$; then the only valid position is $c_k$ for $P_3=\{123,132,312\}$ and $c_{k+1}$ for $P_{18}=\{213,231,321\}$.
    
    \item[Case 2:] A $1$ occurs strictly above and to the right of $(c_l,r)$; then the only valid position is $c_1$ for $P_3=\{123,132,312\}$ and  $c_l$ for $P_{18}=\{213,231,321\}$.
    
    \item[Case 3:] The rectangle with vertices $(1,r+1)$, $(c_1-1,r+1)$, $(c_1-1,\ell(c_2))$, $(1,\ell(c_2))$ contains a $1$; then the only valid position is $c_1$ for $P_3=\{123,132,312\}$ and $c_{l}$ for $P_{18}=\{213,231,321\}$.
\end{itemize}
\end{lemma}

\noindent
{\bf Encoding Process:} We now define an encoding for transversals that avoid pattern set.
Let $T$ be such a transversal of $F$, and let 
$w = w_1 \dots w_n$ be the associated word over $\{0,1,2,\ldots,n\}$.  
For the $1$ in the $i$-th row (from the top), set:  
\begin{itemize}
   \item $w_i=j$ if it is in the $j$-th column of the subboard of remaining white cells and its placement is not forced. 
    \item $w_i = 0$ if the placement is forced.
\end{itemize}

\begin{example}
    We work out an example for the board $(5,5,4,4,3)$. We can see that the transversal $5,3,1,4,2$ avoids $P_{3}$. We obtain the encoding word from the board and transversal. 
\vspace{0.2cm}\\
\scalebox{0.75}{
\ytableausetup{centertableaux} 
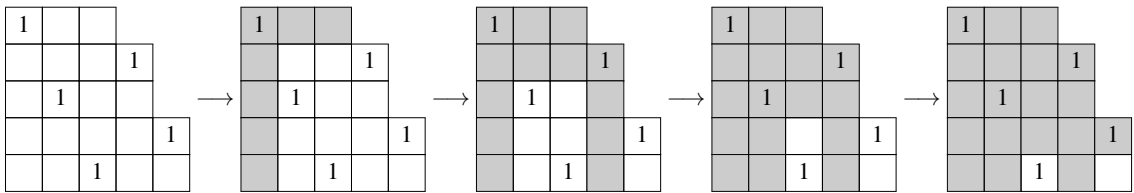
\begin{ytableau}
   $1$ &  & \\
     &  &  & $1$\\
    \textcolor{white}{d} & $1$ & & \textcolor{white}{d}  \\
     &  & \textcolor{white}{d} & & $1$\\
     & \textcolor{white}{d} & $1$& & \\
\end{ytableau}
$\longrightarrow$
\begin{ytableau}
   *(lgrey)$1$ &*(lgrey)  &*(lgrey) \\
     *(lgrey)&  &  & $1$\\
    *(lgrey) & $1$ & & \textcolor{white}{d}  \\
     *(lgrey)&  & \textcolor{white}{d} & & $1$\\
     *(lgrey)& \textcolor{white}{d} & $1$& & \\
\end{ytableau}
$\longrightarrow$
\begin{ytableau}
   *(lgrey)$1$ &*(lgrey)  &*(lgrey) \\
     *(lgrey)& *(lgrey) & *(lgrey) & *(lgrey)$1$\\
    *(lgrey) & $1$ & & *(lgrey) \\
     *(lgrey)&  & \textcolor{white}{d} & *(lgrey)& $1$\\
     *(lgrey)& \textcolor{white}{d} & $1$& *(lgrey)& \\
\end{ytableau}
$\longrightarrow$
\begin{ytableau}
   *(lgrey)$1$ &*(lgrey)  &*(lgrey) \\
     *(lgrey)& *(lgrey) & *(lgrey) & *(lgrey)$1$\\
    *(lgrey) &*(lgrey)$1$ & *(lgrey)& *(lgrey) \\
     *(lgrey)& *(lgrey) & \textcolor{white}{d} & *(lgrey)& $1$\\
     *(lgrey)& *(lgrey) & $1$& *(lgrey)& \\
\end{ytableau}
$\longrightarrow$
\begin{ytableau}
   *(lgrey)$1$ &*(lgrey)  &*(lgrey) \\
     *(lgrey)& *(lgrey) & *(lgrey) & *(lgrey)$1$\\
    *(lgrey) &*(lgrey)$1$ & *(lgrey)& *(lgrey) \\
     *(lgrey)& *(lgrey) & *(lgrey) & *(lgrey)& *(lgrey)$1$\\
     *(lgrey)& *(lgrey) & $1$& *(lgrey)& \\
\end{ytableau}
}\vspace{0.2cm}

In the topmost row, every cell is valid, so $w_1=1$. In the next row, only the second and third cells are valid, yielding $w_2=2$. The placement in the third row is forced, and therefore $w_3=0$. In the next row, both remaining cells are valid, giving $w_4=2$. Finally, the placement in the bottommost row is again forced, so $w_5=0$. Consequently, the encoding word corresponding to the transversal is $w=(1,2,0,2,0).$

The boards below display the insertion process for the transversal avoiding $P_{18}$.\vspace{0.2cm}\\
\scalebox{0.75}{
\ytableausetup{centertableaux} 
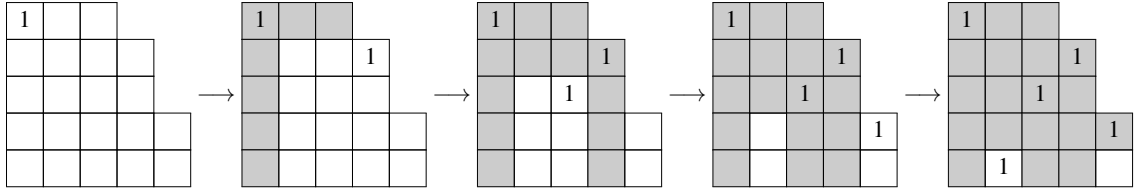
\begin{ytableau}
   $1$  &  & \\
     &  & & \\
    \textcolor{white}{d} & & & \textcolor{white}{d}   \\
     &  & \textcolor{white}{d} & & \\
     & \textcolor{white}{d} & \textcolor{white}{d} & & \\
\end{ytableau}
$\longrightarrow$
\begin{ytableau}
   *(lgrey)$1$ & *(lgrey) & *(lgrey)\\
     *(lgrey)&  & & $1$\\
    *(lgrey) & & & \textcolor{white}{d}   \\
     *(lgrey)&  & \textcolor{white}{d} & & \\
     *(lgrey)& \textcolor{white}{d} & \textcolor{white}{d} & & \\
\end{ytableau}
$\longrightarrow$
\begin{ytableau}
   *(lgrey)$1$ & *(lgrey) & *(lgrey)\\
     *(lgrey)& *(lgrey) &*(lgrey) & *(lgrey)$1$\\
    *(lgrey) & & $1$ & *(lgrey)\\
     *(lgrey)&  &  & *(lgrey)& \\
     *(lgrey)& \textcolor{white}{d} &  & *(lgrey)& \\
\end{ytableau}
$\longrightarrow$
\begin{ytableau}
   *(lgrey)$1$ & *(lgrey) & *(lgrey)\\
     *(lgrey)& *(lgrey) &*(lgrey) & *(lgrey)$1$\\
    *(lgrey) &*(lgrey) & *(lgrey)$1$ & *(lgrey)\\
     *(lgrey)&  & *(lgrey) & *(lgrey)& $1$\\
     *(lgrey)& \textcolor{white}{d} & *(lgrey) & *(lgrey)& \\
\end{ytableau}
$\longrightarrow$
\begin{ytableau}
   *(lgrey)$1$ & *(lgrey) & *(lgrey)\\
     *(lgrey)& *(lgrey) &*(lgrey) & *(lgrey)$1$\\
    *(lgrey) &*(lgrey) & *(lgrey)$1$ & *(lgrey)\\
     *(lgrey)& *(lgrey) & *(lgrey) & *(lgrey)& *(lgrey)$1$\\
     *(lgrey)& $1$& *(lgrey) & *(lgrey)& \\
\end{ytableau}
}\vspace{0.2cm} 

The condition $w_1=2$ indicates that, in the topmost row, the entry $1$ is placed in the first valid column. In the second row, only the second and third columns are valid. Since $w_2=2$, the entry $1$ is placed in the third column. The placement in the third row is then forced, and the entry $1$ must occupy the second column.

In the fourth row, both remaining cells are valid. The condition $w_4=2$ therefore places the entry $1$ in the second available cell. Finally, the placement in the last row is forced, as only one white cell remains.
Consequently, the resulting transversal corresponds to the permutation $5,1,3,4,2$, which avoids $P_{18}$.
\end{example}
\subsection{Class IV}\label{tripIV}

This class has the sets $P_2,P_{19}$. These two were proven to be shape-Wilf-equivalent in 
\cite[Theorem 7]{SUK-arxiv}.

\begin{theorem}[{\cite[Theorem $7$]{SUK-arxiv}}]
 $P_{2}=\{123,132,231\}\sim_s P_{19}=\{213,312,321\}.$  
\end{theorem}
\begin{remark}
When $n \geq 8$, the remaining seven are singleton classes (V--XI).  This can be seen 
from Table \ref{table1}: as the numbers are different, they are different classes.
\end{remark}

\section{Patterns in Perfect matchings}
\label{sec:perfect_matchings}
For each positive integer $n$, let $\mathcal{F}_n$ denote the family of Ferrers boards with $n$ rows and $n$ columns that contain the staircase board. Given a Ferrers board $F\in\mathcal{F}_n$, we write $\mathcal{T}_F$ for the set of all transversals of $F$, and $\mathcal{T}_F(P)$ for the subset consisting of those transversals that avoid every pattern in the set $P$. Define
$\displaystyle
\mathcal{T}_n=\bigcup_{F\in\mathcal{F}_n}\mathcal{T}_F,$
and $\mathcal{T}_n(P)=\bigcup_{F\in\mathcal{F}_n}\mathcal{T}_F(P).$

Let $\mathcal{M}_n$ be the collection of perfect matchings on the vertex set $[2n]$. 
For a Ferrers board $F$, the associated set of perfect matchings is denoted by $\mathcal{M}_F$, and $\mathcal{M}_F(P)$ denotes the subset of perfect matchings avoiding the pattern set $P$. Their results (see Section~2.2 of \cite{matching_partition}) show that
$|\mathcal{M}_F(p)|=|\mathcal{T}_F(p)|$
for every single pattern $p$. Consequently,
$|\mathcal{M}_F(P)|=|\mathcal{T}_F(P)|$
holds for every collection of patterns $P$. Hence the study of shape-Wilf-equivalence may equivalently be carried out in the setting of restricted perfect matchings; specifically,
$P_1\sim_s P_2
\quad \Longleftrightarrow \quad
|\mathcal{M}_F(P_1)|=|\mathcal{M}_F(P_2)|$
for all Ferrers boards $F$.

Define
$A_n(P)=\{(T,F): F\in\mathcal{F}_n \text{ and } T\in\mathcal{T}_F(P)\},$
and let $a_n:=|A_n(P)|.$
By the correspondence above, $a_n$ also counts the number of perfect matchings in $\mathcal{M}_n$ avoiding $P$, i.e.,
\begin{equation}\label{M_n(P)=a_n}
    |\mathcal{M}_n(P)|=a_n.
\end{equation}
We will use the notation ``$a_n$'' repeatedly in theorems. 
In particular, if two pattern sets $P_1$ and $P_2$ are shape-Wilf-equivalent, then $|A_n(P_1)|=|A_n(P_2)|$ for every $n$.

\begin{figure}[H]

   \centering
{
\begin{minipage}{.3\textwidth}

\begin{tikzpicture}[scale=.55]
	
        \filldraw  (0,0) circle (2pt); 
        \filldraw  (1,0) circle (2pt); 
        \filldraw  (2,0) circle (2pt); 
        \filldraw  (6,0) circle (2pt); 
        \filldraw  (5,0) circle (2pt); 
        \filldraw  (4,0) circle (2pt);

        \draw (0,0) arc (180:0:2) ;
        \draw (1,0) arc (180:0:2) ;
        \draw (2,0) arc (180:0:2) ;
        \node[] at (0,-0.4)   {  \tiny $1$};
        \node[] at (1,-0.4)   {  \tiny $2$};
        \node[] at (2,-0.4)   {  \tiny $3$};
        \node[] at (4,-0.4)   { \tiny {$4$}};
        \node[] at (5,-0.4)   {  \tiny{$5$}};
         \node[] at(6,-0.4)   {\tiny $6$};
         \node[] at(3,-1.25)   { $321$};
        
\end{tikzpicture}
\end{minipage}
\begin{minipage}{.3\textwidth}
  \begin{tikzpicture}[scale=.55]
       \filldraw  (0,0) circle (2pt); 
        \filldraw  (1,0) circle (2pt); 
        \filldraw  (2,0) circle (2pt); 
        \filldraw  (3,0) circle (2pt); 
        \filldraw  (5,0) circle (2pt); 
        \filldraw  (4,0) circle (2pt);

        \draw (0,0) arc (180:0:2.5) ;
        \draw (1,0) arc (180:0:1.5) ;
        \draw (2,0) arc (180:0:.5) ;
        
        \node[] at (0,-0.4)   {  \tiny $1$};
        \node[] at (1,-0.4)   {  \tiny $2$};
        \node[] at (2,-0.4)   {  \tiny $3$};
        \node[] at(3,-0.4)   {\tiny $4$};
        \node[] at (4,-0.4)   { \tiny {$5$}};
        \node[] at (5,-0.4)   {  \tiny{$6$}};
         
         \node[] at(2.5,-1.25)   { $123$};
  \end{tikzpicture}  
\end{minipage}
\begin{minipage}{.3\textwidth}
  \begin{tikzpicture}[scale=.556]
       \filldraw  (0,0) circle (2pt); 
        \filldraw  (.75,0) circle (2pt); 
        \filldraw  (1.75,0) circle (2pt); 
        \filldraw  (3.25,0) circle (2pt); 
        \filldraw  (5,0) circle (2pt); 
        \filldraw  (4.25,0) circle (2pt);

        \draw (0,0) arc (180:0:2.5) ;
        \draw (.75,0) arc (180:0:1.25) ;
        \draw (1.75,0) arc (180:0:1.25) ;
        
        \node[] at (0,-0.4)   {  \tiny $1$};
        \node[] at (.75,-0.4)   {  \tiny $2$};
        \node[] at (1.75,-0.4)   {  \tiny $3$};
        \node[] at(3.25,-0.4)   {\tiny $4$};
        \node[] at (4.25,-0.4)   { \tiny {$5$}};
        \node[] at (5,-0.4)   {  \tiny{$6$}};
         
         \node[] at(2.5,-1.25)   { $132$};
  \end{tikzpicture}  
\end{minipage}

\begin{minipage}{.3\textwidth}

\begin{tikzpicture}[scale=.556]
	
        \filldraw  (0,0) circle (2pt); 
        \filldraw  (1,0) circle (2pt); 
        \filldraw  (2.5,0) circle (2pt); 
        \filldraw  (6.5,0) circle (2pt); 
        \filldraw  (5,0) circle (2pt); 
        \filldraw  (4,0) circle (2pt);

        \draw (0,0) arc (180:0:2.5) ;
        \draw (1,0) arc (180:0:1.5) ;
        \draw (2.5,0) arc (180:0:2) ;
        \node[] at (0,-0.4)   {  \tiny $1$};
        \node[] at (1,-0.4)   {  \tiny $2$};
        \node[] at (2.5,-0.4)   {  \tiny $3$};
        \node[] at (4,-0.4)   { \tiny {$4$}};
        \node[] at (5,-0.4)   {  \tiny{$5$}};
         \node[] at(6.5,-0.4)   {\tiny $6$};
         \node[] at(3.5,-1.25)   { $231$};
       
\end{tikzpicture}
\end{minipage}
\begin{minipage}{.3\textwidth}
\begin{tikzpicture}[scale=.52]

        \filldraw  (0,0) circle (2pt); 
        \filldraw  (1.5,0) circle (2pt); 
        \filldraw  (2.5,0) circle (2pt); 
        \filldraw  (6.5,0) circle (2pt); 
        \filldraw  (5.5,0) circle (2pt); 
        \filldraw  (4,0) circle (2pt);

        \draw (1.5,0) arc (180:0:2.5) ;
        \draw (2.5,0) arc (180:0:1.5) ;
        \draw (0,0) arc (180:0:2) ;
        \node[] at (0,-0.4)   {  \tiny $1$};
        \node[] at (1.5,-0.4)   {  \tiny $2$};
        \node[] at (2.5,-0.4)   {  \tiny $3$};
        \node[] at (4,-0.4)   { \tiny {$4$}};
        \node[] at (5.5,-0.4)   {  \tiny{$5$}};
         \node[] at(6.5,-0.4)   {\tiny $6$};
         \node[] at(3.5,-1.25)   { $312$};
       
\end{tikzpicture}
\end{minipage}
\begin{minipage}{.3\textwidth}
    \begin{tikzpicture}[scale=.52]

        \filldraw  (0,0) circle (2pt); 
        \filldraw  (1.25,0) circle (2pt); 
        \filldraw  (2.15,0) circle (2pt); 
        \filldraw  (6.25,0) circle (2pt); 
        \filldraw  (5,0) circle (2pt); 
        \filldraw  (4.15,0) circle (2pt);

        \draw (0,0) arc (180:0:2.5) ;
        \draw (1.25,0) arc (180:0:2.5) ;
        \draw (2.15,0) arc (180:0:1) ;
        \node[] at (0,-0.4)   {  \tiny $1$};
        \node[] at (1.25,-0.4)   {  \tiny $2$};
        \node[] at (2.15,-0.4)   {  \tiny $3$};
        \node[] at (4.15,-0.4)   { \tiny {$4$}};
        \node[] at (5,-0.4)   {  \tiny{$5$}};
         \node[] at(6.25,-0.4)   {\tiny $6$};
         \node[] at(3.25,-1.25)   { $213$};
       
\end{tikzpicture}
\end{minipage}
}
\caption{Patterns in perfect matchings corresponding to permutation patterns of length $3$}  
\label{Fig1}
 \end{figure}

Perfect matchings avoiding certain triples of patterns have been enumerated. For example, Cervetti and Ferrari \cite[Corollary 3.2]{enumeration_matching} have enumerated the number of perfect matchings avoiding $\{132,231,312\}$. Their result is the following.   
\begin{theorem}[Cervetti and Ferrari]
\label{class2}
    For $P_{14}=\{132,231,312\}$, 
    $\sum_{n \geq 0/1?} a_nx^n = \frac{1}{1-xC(x)C(C(x)-1)}$ where $C(x)$ is the Catalan generating function.
\end{theorem}

\begin{remark}
Theorem \ref{thm:11-classes} partitions triple pattern sets by shape-Wilf-equivalence. By Bloom and Elizalde \cite[Section 2.2]{matching_partition}, shape-Wilf-equivalent pattern sets are avoided by the same number of perfect matchings. 
Thus, 
we only consider a single representative from each equivalence class of Table 
\ref{table1}. 
\end{remark}

We begin with some notation and definitions.
\begin{definition}
    A perfect matching on $2n$ vertices is said to be \emph{disconnected} if there exists an index 
    $i<2n-2$ such that no arc with  endpoint $j>i$ has its  opening point $j'\le i$; 
    equivalently, the matching can be expressed as the union of two disjoint perfect matchings. Otherwise, we call it \emph{connected}.
\end{definition}

\subsection{Class VI}
\begin{theorem}\label{linear_triple4}
    For $P_4=\{123,132,321\}$, when $n\geq 5$, $a_n$ satisfies the recurrence relation 
    $$ a_n=4a_{n-1}+2a_{n-2}-4a_{n-3}+a_{n-4} \mbox{ with } a_1=1, a_2=3, a_3=12, a_4=51.$$
\end{theorem}
\begin{proof}
    We prove the result by induction on $n$. Given a matching on $2(n-1)$ or $2(n-2)$ vertices,
we obtain a matching on $2n$ vertices by adding either a single new rightmost opening arc or a pair of rightmost opening arcs. Since the perfect matchings are required to avoid the patterns $123$, $132$, and $321$, it follows from Figure \ref{Fig1} that they cannot contain three nested arcs (corresponding to $123$), three crossing arcs (corresponding to $321$), or a crossing under a nesting (corresponding to $132$). 
Consequently, to ensure that the resulting matching avoids the pattern set $P_4$, only certain placements of the final arc (or the final two arcs) are allowed.
To add either a single arc or a pair of arcs, we relabel the $2(n-1)$ or $2(n-2)$ many vertices accordingly so that the new rightmost arc(s) can be properly positioned on $2n$ vertices. These admissible configurations are illustrated in Figure \ref{Fig6}. For simplicity, in Figure \ref{Fig6} we denote the vertex $2n-i$ by $-i$, for $1\le i\le 2n-1$.  
 \begin{figure}[!htbp]
    \begin{minipage}{.24\textwidth}
    \begin{tikzpicture}[scale=.45]

        \filldraw  (0,0) circle (2pt); 
        \filldraw  (7,0) circle (2pt); 
        \filldraw  (5,0) circle (2pt); 
        \filldraw  (4,0) circle (2pt); 
        \draw (0,0) arc (180:0:2) ;
        \draw[dotted][thick] (5,0) arc (180:0:1) ;
        \node[] at (0,-0.3)   { \tiny $i$};
        \node[] at (3.8,-0.3)   {\tiny $-2$};
         \node[] at (4.8,-0.3)   {\tiny $-1$};
         \node[] at (7,-0.3)   { \tiny$2n$};
         \node[] at (4,-1.25)   { $M_1$};

\end{tikzpicture}
    \end{minipage}
     \begin{minipage}{.23\textwidth}
     \begin{tikzpicture}[scale=.45]
	
        \filldraw  (0,0) circle (2pt); 
        \filldraw  (6.5,0) circle (2pt); 
        \filldraw  (5,0) circle (2pt); 
        \filldraw  (3.5,0) circle (2pt); 
        
        \draw (0,0) arc (180:0:2.5) ;
        \draw[dotted][thick] (3.5,0) arc (180:0:1.5) ;
        \node[] at (0,-0.3)   { \tiny$i$};
        \node[] at (3.3,-0.3)   { \tiny $-2$};
         \node[] at (4.8,-0.3)   { \tiny $-1$};
         \node[] at (6.5,-0.3)   { \tiny $2n$};
        
         \node[] at (4,-1.25)   { $M_2$};
\end{tikzpicture}
\end{minipage}
\begin{minipage}{.2\textwidth}
     \begin{tikzpicture}[scale=.41]

      \filldraw  (0,0) circle (2pt); 
        \filldraw  (6,0) circle (2pt); 
        \filldraw  (5,0) circle (2pt); 
        \filldraw  (3,0) circle (2pt); 
        \draw (0,0) arc (180:0:3) ;
        \draw[dotted][thick] (3,0) arc (180:0:1) ;
        \node[] at (0,-0.3)   { \tiny$i$};
        \node[] at (2.8,-0.3)   { \tiny$-2$};
         \node[] at (4.8,-0.3)   { \tiny$-1$};
         \node[] at (6,-0.3)   { \tiny$2n$};
         \node[] at (3,-1.25)   { $M_3$};
        
     \end{tikzpicture}
\end{minipage}
\begin{minipage}{.3\textwidth}
     \begin{tikzpicture}[scale=.5]
      \filldraw  (-0.5,0) circle (2pt); 
        \filldraw  (8,0) circle (2pt); 
        \filldraw  (7,0) circle (2pt); 
        \filldraw  (3,0) circle (2pt); 
        \filldraw  (5.5,0) circle (2pt); 
        \filldraw  (4,0) circle (2pt);

        \draw (-0.5,0) arc (180:0:3) ;
        \draw  (3,0) arc (180:0:2.5) ;
        \draw[dotted][thick] (4,0) arc (180:0:1.5) ;
        \node[] at (-0.5,-0.25)   { \tiny$i$};
       
        \node[] at (2.8,-0.3)   {\tiny {$j$}};
        \node[] at (3.8,-0.3)   {\tiny {$-3$}};
        \node[] at (5.3,-0.3)   {\tiny {$-2$}};
         \node[] at (6.8,-0.3)   {\tiny $-1$};
         \node[] at (8,-0.25)   {\tiny $2n$};
        \node[] at (4,-1.25)   { $M_4$};
       
\end{tikzpicture}
\end{minipage}

\begin{minipage}{.3\textwidth}
     \begin{tikzpicture}[scale=.55]
       \filldraw  (0,0) circle (2pt); 
        \filldraw  (4.5,0) circle (2pt); 
        \filldraw  (5,0) circle (2pt); 
        \filldraw  (2.5,0) circle (2pt); 
        \filldraw  (7,0) circle (2pt);
        \filldraw  (2,0) circle (2pt); 
        \draw (0,0) arc (180:0:2.5) ;
        \draw (2,0) arc (180:0:2.5) ;
        \draw[dotted][thick] (2.5,0) arc (180:0:1) ;
        \node[] at (0,-0.25)   { \tiny$i$};
        \node[] at (1.8,-0.25)   { \tiny$j$};
        \node[] at (2.3,-0.3)   { \tiny$-3$};
         \node[] at (4.3,-0.3)   { \tiny$-2$};
         \node[] at (4.9,-0.3)   { \tiny$-1$};
         \node[] at (7,-0.25)   { \tiny$2n$};

         \node[] at (3.5,-1.25)   { $M_5$};
     \end{tikzpicture}
\end{minipage}
\begin{minipage}{.3\textwidth}
     \begin{tikzpicture}[scale=.55]
      \filldraw  (0,0) circle (2pt); 
        \filldraw  (1,0) circle (2pt); 
        \filldraw  (7,0) circle (2pt); 
        \filldraw  (6,0) circle (2pt); 
        \filldraw  (5,0) circle (2pt); 
        \filldraw  (4,0) circle (2pt);

        \draw (0,0) arc (180:0:3) ;
        \draw (1,0) arc (180:0:2) ;
        \draw[dotted][thick] (4,0) arc (180:0:1.5) ;
        \node[] at (0,-0.3)   {  \tiny $i$};
        \node[] at (1,-0.3)   {  \tiny $j$};
        \node[] at (3.8,-0.3)   { \tiny {$-3$}};
        \node[] at (4.8,-0.3)   {  \tiny{$-2$}};
         \node[] at (5.8,-0.3)   {\tiny $-1$};
         \node[] at (7,-0.3)   { \tiny $2n$};
         \node[] at (3,-1.25)   { $M_6$};
     \end{tikzpicture}
\end{minipage}
\begin{minipage}{.38\textwidth}
     \begin{tikzpicture}[scale=.6]
 \filldraw  (-0.5,0) circle (2pt);
        \filldraw  (.5,0) circle (2pt);
        \filldraw  (8.25,0) circle (2pt); 
        \filldraw  (7.25,0) circle (2pt); 
        \filldraw  (2.75,0) circle (2pt); 
        \filldraw  (5.5,0) circle (2pt); 
        \filldraw  (4.5,0) circle (2pt); 
        \filldraw  (3.75,0) circle (2pt);

        \draw (-0.5,0) arc (180:0:3) ;
        \draw (.5,0) arc (180:0:2) ;
        \draw[dotted][thick] (2.75,0) arc (180:0:2.75) ;
        \draw[dotted][thick] (3.75,0) arc (180:0:1.75) ;
        
        \node[] at (-0.5,-0.3)   { \tiny$i$};
         \node[] at (0.5,-0.3)   { \tiny$j$};
        \node[] at (2.55,-0.3)   {\tiny {$-5$}};       
        \node[] at (3.55,-0.3)   {\tiny {$-4$}};
        \node[] at (4.35,-0.3)   {\tiny {$-3$}};
        \node[] at (5.35,-0.3)   {\tiny {$-2$}};
         \node[] at (7.05,-0.3)   {\tiny $-1$};
         \node[] at (8.25,-0.3)   {\tiny $2n$};
     \node[] at (3,-1.25)   { $M_7$};
\end{tikzpicture}
\end{minipage}

\caption{Permitted perfect matchings}  
\label{Fig6}
 \end{figure}
Let $m_i(n)$ denote the number of perfect matchings of type $M_i$ on $2n$ vertices,
for $i=1,2,\ldots,6$, and let $a_n$ be the total number of $P_4$-avoiding perfect matchings
on $2n$ vertices.
Perfect matchings of types $M_1$, $M_2$, and $M_3$ can be obtained by adding a single arc
to any $P_4$-avoiding matching on $2(n-1)$ vertices, without imposing any additional
conditions. Hence, $m_1(n)=m_2(n)=m_3(n)=a_{n-1}.$

Perfect matchings of type $M_4$ and $M_5$ on $2n$ vertices can be constructed from perfect matchings of type $M_2$, $M_5$ and $M_6$ on $2(n-1)$ vertices by adding a rightmost opening arc nested under the rightmost closer and crossing the second rightmost closer. Thus,
\begin{equation}\label{eq6}
  m_4(n)=m_5(n)=m_2(n-1)+m_5(n-1)+m_6(n-1).
\end{equation}

Perfect matchings of type $M_6$ on $2n$ vertices can be constructed from perfect matchings of type $M_3$, $M_4$ and $M_7$ on $2(n-1)$ vertices by adding a rightmost opening arc  crossing the two rightmost closers. Thus,
\begin{equation}\label{eq7}
  m_6(n)=m_3(n-1)+m_4(n-1)+m_7(n-1).
\end{equation}
Lastly, a matching of type $M_7$ is identical to $M_6$, but instead of adding a single rightmost opening arc, we have to add two rightmost opening arcs. Hence, 
\begin{equation}\label{eq8}
  m_7(n)=m_6(n-1).
\end{equation}
Using Equations \eqref{eq7} and \eqref{eq8}  and adding Equations 
\eqref{eq6} and \eqref{eq7} we have,
\begin{eqnarray}
  m_6(n) & = & a_{n-2}+m_4(n-1)+m_6(n-2). \label{eq9} \\
     m_4(n)+m_6(n)& = & a_{n-1}-m_1(n-1)=a_{n-1}-a_{n-2}.\label{eq10}
 \end{eqnarray}
Summing over all types, we obtain
\begin{equation}\label{eq11}
  a_n=\sum_{i=1}^{7} m_i(n)=3a_{n-1}+2m_4(n)+m_6(n)+m_6(n-1).  
\end{equation}
Combining \eqref{eq9}, \eqref{eq10} and \eqref{eq11} we get,
$a_n = 4a_{n-1} + 2a_{n-2} - 4a_{n-3} + a_{n-4}$, 
 completing the proof.  
\end{proof}

\subsection{Class VII}
\begin{theorem}\label{linear_triple1}
    For $P_7=\{123,213,321\}$, when $n\geq 3$,  $a_n$ satisfies the recurrence 
    relation $a_{n}=5a_{n-1}-3a_{n-2}$ with $a_0=1$, $a_1=1$, $a_2=3$.
\end{theorem}
\begin{proof}
    We prove the result by induction on $n$. Given a matching on $2(n-1)$ or $2(n-2)$ vertices,
we obtain a matching on $2n$ vertices by adding either a single new rightmost opening arc or a pair of rightmost opening arcs. Since the perfect matchings are required to avoid the patterns $123$, $213$, and $321$, it follows from Figure \ref{Fig1} that they cannot contain three nested arcs (corresponding to $123$), three crossing arcs (corresponding to $321$), or a nesting under a crossing (corresponding to $213$). 
Consequently, to ensure that the resulting matching avoids the pattern set $P_7$, only certain placements of the final arc (or the final two arcs) are allowed.
To add either a single arc or a pair of arcs, we relabel the $2(n-1)$ or $2(n-2)$ many vertices accordingly so that the new rightmost arc(s) can be properly positioned on $2n$ vertices. These admissible configurations are illustrated in Figure \ref{Fig2}. For simplicity, in Figure \ref{Fig2} we denote the vertex $2n-i$ by $-i$, for $1\le i\le 2n-1$.
\begin{figure}[!htbp]
    \begin{minipage}{.2\textwidth}
    \begin{tikzpicture}[scale=.40]
        \filldraw  (0,0) circle (2pt); 
        \filldraw  (7,0) circle (2pt); 
        \filldraw  (5,0) circle (2pt); 
        \filldraw  (4,0) circle (2pt); 
        \draw (0,0) arc (180:0:2) ;
        \draw[dotted][thick] (5,0) arc (180:0:1) ;
        \node[] at (0,-0.3)   { \tiny $i$};
        \node[] at (3.8,-0.3)   {\tiny $-2$};
         \node[] at (4.8,-0.3)   {\tiny $-1$};
         \node[] at (7,-0.3)   { \tiny$2n$};
         \node[] at (4,-1.25)   { $M_1$};
\end{tikzpicture}
   \end{minipage}
\begin{minipage}{.18\textwidth}
     \begin{tikzpicture}[scale=.4]
        \filldraw  (0,0) circle (2pt); 
        \filldraw  (6.5,0) circle (2pt); 
        \filldraw  (5,0) circle (2pt); 
        \filldraw  (3.5,0) circle (2pt); 
        
        \draw (0,0) arc (180:0:2.5) ;
        \draw[dotted][thick] (3.5,0) arc (180:0:1.5) ;
        \node[] at (0,-0.3)   { \tiny$i$};
        \node[] at (3.3,-0.3)   { \tiny $-2$};
         \node[] at (4.8,-0.3)   { \tiny $-1$};
         \node[] at (6.5,-0.3)   { \tiny $2n$};
        
         \node[] at (4,-1.25)   { $M_2$};
    \end{tikzpicture}
    \end{minipage}
\begin{minipage}{.3\textwidth}
     \begin{tikzpicture}[scale=.5]
 \filldraw  (-0.5,0) circle (2pt);
        \filldraw  (.5,0) circle (2pt);
        \filldraw  (8.25,0) circle (2pt); 
        \filldraw  (7.25,0) circle (2pt); 
        \filldraw  (2.75,0) circle (2pt); 
        \filldraw  (5.5,0) circle (2pt); 
        \filldraw  (4.5,0) circle (2pt); 
        \filldraw  (3.75,0) circle (2pt);

        \draw (-0.5,0) arc (180:0:3) ;
        \draw (.5,0) arc (180:0:2) ;
        \draw[dotted][thick] (2.75,0) arc (180:0:2.75) ;
        \draw[dotted][thick] (3.75,0) arc (180:0:1.75) ;
        
        \node[] at (-0.5,-0.3)   { \tiny$i$};
         \node[] at (0.5,-0.3)   { \tiny$j$};
        \node[] at (2.55,-0.3)   {\tiny {$-5$}};       
        \node[] at (3.55,-0.3)   {\tiny {$-4$}};
        \node[] at (4.35,-0.3)   {\tiny {$-3$}};
        \node[] at (5.35,-0.3)   {\tiny {$-2$}};
         \node[] at (7.05,-0.3)   {\tiny $-1$};
         \node[] at (8.25,-0.3)   {\tiny $2n$};
     \node[] at (3,-1.25)   { $M_7$};
\end{tikzpicture}
\end{minipage}
\begin{minipage}{.27\textwidth}
     \begin{tikzpicture}[scale=.5]
     \filldraw  (-0.5,0) circle (2pt);
        \filldraw  (1.25,0) circle (2pt);
        \filldraw  (8.25,0) circle (2pt); 
        \filldraw  (6.5,0) circle (2pt); 
        \filldraw  (2.75,0) circle (2pt); 
        \filldraw  (5.5,0) circle (2pt); 
        \filldraw  (4.5,0) circle (2pt); 
        \filldraw  (3.75,0) circle (2pt);

        \draw (-0.5,0) arc (180:0:3) ;
        \draw (1.25,0) arc (180:0:1.25) ;
        \draw[dotted][thick] (2.75,0) arc (180:0:2.75) ;
        \draw[dotted][thick] (4.5,0) arc (180:0:1) ;
        
        \node[] at (-0.5,-0.3)   { \tiny$i$};
         \node[] at (1.25,-0.3)   { \tiny$j$};
        \node[] at (2.55,-0.3)    {\tiny {$-5$}};       
        \node[] at (3.5,-0.3)   {\tiny {$-4$}};
        \node[] at (4.35,-0.3)   {\tiny {$-3$}};
        \node[] at (5.35,-0.3)   {\tiny {$-2$}};
         \node[] at (6.35,-0.3)   {\tiny $-1$};
         \node[] at (8.25,-0.3)   {\tiny $2n$};
     \node[] at (3,-1.25)   { $M_8$};
     \end{tikzpicture}
\end{minipage}

\begin{minipage}{.27\textwidth}
     \begin{tikzpicture}[scale=.4]
      \filldraw  (-0.5,0) circle (2pt); 
        \filldraw  (8,0) circle (2pt); 
        \filldraw  (7,0) circle (2pt); 
        \filldraw  (3,0) circle (2pt); 
        \filldraw  (5.5,0) circle (2pt); 
        \filldraw  (4,0) circle (2pt); 
        \draw (-0.5,0) arc (180:0:3) ;
        \draw[dotted][thick] (3,0) arc (180:0:2.5) ;
        \draw[dotted][thick] (4,0) arc (180:0:1.5) ;
        \node[] at (-0.5,-0.25)   { \tiny$i$};   
        \node[] at (2.8,-0.3)   {\tiny {$-4$}};
        \node[] at (3.8,-0.3)   {\tiny {$-3$}};
        \node[] at (5.3,-0.3)   {\tiny {$-2$}};
         \node[] at (6.8,-0.3)   {\tiny $-1$};
         \node[] at (8,-0.25)   {\tiny $2n$};
        \node[] at (4,-1.25)   { $M_3$};
\end{tikzpicture}
\end{minipage}
\begin{minipage}{.23\textwidth}
     \begin{tikzpicture}[scale=.4]
      \filldraw  (0,0) circle (2pt); 
        \filldraw  (6,0) circle (2pt); 
        \filldraw  (5,0) circle (2pt); 
        \filldraw  (3,0) circle (2pt); 
        \draw (0,0) arc (180:0:3) ;
        \draw[dotted][thick] (3,0) arc (180:0:1) ;
        \node[] at (0,-0.3)   { \tiny$i$};
        \node[] at (2.8,-0.3)   { \tiny$-2$};
         \node[] at (4.8,-0.3)   { \tiny$-1$};
         \node[] at (6,-0.3)   { \tiny$2n$};
         \node[] at (3,-1.25)   { $M_4$};
        
     \end{tikzpicture}
\end{minipage}
\begin{minipage}{.24\textwidth}
     \begin{tikzpicture}[scale=.4]
      \filldraw  (0,0) circle (2pt); 
        \filldraw  (6,0) circle (2pt); 
        \filldraw  (5,0) circle (2pt); 
        \filldraw  (4,0) circle (2pt); 
        \filldraw  (7,0) circle (2pt);
        \filldraw  (2,0) circle (2pt); 
        \draw (0,0) arc (180:0:3.5) ;
        \draw (2,0) arc (180:0:1.5) ;
        \draw[dotted][thick] (4,0) arc (180:0:1) ;
        \node[] at (0,-0.25)   { \tiny$i$};
        \node[] at (2,-0.25)   { \tiny$j$};
        \node[] at (3.8,-0.3)   { \tiny$-3$};
         \node[] at (4.8,-0.3)   { \tiny$-2$};
         \node[] at (5.8,-0.3)   { \tiny$-1$};
         \node[] at (7,-0.25)   { \tiny$2n$};

         \node[] at (3,-1.25)   { $M_5$};
     \end{tikzpicture}
\end{minipage}
\begin{minipage}{.24\textwidth}
     \begin{tikzpicture}[scale=.4]
      \filldraw  (0,0) circle (2pt); 
        \filldraw  (1,0) circle (2pt); 
        \filldraw  (7,0) circle (2pt); 
        \filldraw  (6,0) circle (2pt); 
        \filldraw  (5,0) circle (2pt); 
        \filldraw  (4,0) circle (2pt);

        \draw (0,0) arc (180:0:3) ;
        \draw (1,0) arc (180:0:2) ;
        \draw[dotted][thick] (4,0) arc (180:0:1.5) ;
        \node[] at (0,-0.3)   {  \tiny $i$};
        \node[] at (1,-0.3)   {  \tiny $j$};
        \node[] at (3.8,-0.3)   { \tiny {$-3$}};
        \node[] at (4.8,-0.3)   {  \tiny{$-2$}};
         \node[] at (5.8,-0.3)   {\tiny $-1$};
         \node[] at (7,-0.3)   { \tiny $2n$};
         \node[] at (3,-1.25)   { $M_6$};
     \end{tikzpicture}
\end{minipage}

\caption{Permitted perfect matchings}  
\label{Fig2}
 \end{figure}
Let $m_i(n)$ denote the number of perfect matchings of type $M_i$ on $2n$ vertices,
for $i=1,2,\ldots,8$, and let $a_n$ be the total number of $P_7$-avoiding perfect matchings
on $2n$ vertices.

Perfect matchings of types $M_1$, $M_2$, and $M_4$ can be obtained by adding a single arc
to any $P_7$-avoiding matching on $2(n-1)$ vertices, without imposing any 
additional conditions. Hence $m_1(n)=m_2(n)=m_4(n)=a_{n-1}.$
Similarly, perfect matchings of type $M_3$ are obtained by adding two arcs to any
$P_7$-avoiding matching on $2(n-2)$ vertices, and therefore $m_3(n)=a_{n-2}.$

To construct perfect matchings of types $M_5$ and $M_6$, we must add a single arc to a
matching on $2(n-1)$ vertices that ends with two nested arcs. The perfect matchings that
end with a single arc are precisely those of types $M_1$, $M_2$, and $M_6$,
and hence they cannot contribute to $M_5$ or $M_6$. Consequently,
$m_5(n)=m_6(n)=a_{n-1}-m_1(n-1)-m_2(n-1)-m_6(n-1)
        =a_{n-1}-2a_{n-2}-m_6(n-1).$
Finally, perfect matchings of types $M_7$ and $M_8$ are obtained by adding two arcs to a
matching on $2(n-2)$ vertices that ends with two nested arcs, and therefore
$m_7(n)=m_8(n)=m_6(n-1).$
Summing over all types, we obtain $
a_n=\sum_{i=1}^{8} m_i(n)=5a_{n-1}-3a_{n-2},$
which completes the proof.
\end{proof}

\subsection{Class VIII}
\begin{theorem}\label{linear_triple2}
    For $P_9=\{123,231,321\}$, when $n\geq 4$, $a_n$ satisfies the recurrence relation 
     $$a_{n}=6a_{n-1}-8a_{n-2}+2a_{n-3} \mbox{ with } a_0=1, a_1=1, a_2=3, a_3=12.$$
\end{theorem}
\begin{proof}
    We prove the result by induction on $n$. Given a matching on $2(n-1)$ vertices,
we obtain a matching on $2n$ vertices by adding either a single new rightmost opening arc. Since the perfect matchings are required to avoid the patterns $123$, $231$ and $321$, it follows from Figure \ref{Fig1} that they cannot contain three nested arcs (corresponding to $123$), three crossing arcs (corresponding to $321$), nor can it contain two nestings together with  a right crossing (corresponding to $231$).
Consequently, to ensure that the resulting matching avoids the pattern set $P_9$, only certain placements of the final arc are allowed.
To add either a single arc, we relabel the $2(n-1)$ vertices accordingly so that the new rightmost arc can be properly positioned on $2n$ vertices. These admissible configurations are illustrated in Figure \ref{Fig3}. For simplicity, in Figure \ref{Fig3} we denote the vertex $2n-i$ by $-i$, for $1\le i\le 2n-1$.
 \begin{figure}[!htbp]
    \begin{minipage}{.23\textwidth}
    \begin{tikzpicture}[scale=.47]

        \filldraw  (0,0) circle (2pt); 
        \filldraw  (7,0) circle (2pt); 
        \filldraw  (5,0) circle (2pt); 
        \filldraw  (4,0) circle (2pt); 
        \draw (0,0) arc (180:0:2) ;
        \draw[dotted][thick] (5,0) arc (180:0:1) ;
        \node[] at (0,-0.3)   { \tiny $i$};
        \node[] at (3.8,-0.3)   {\tiny $-2$};
         \node[] at (4.8,-0.3)   {\tiny $-1$};
         \node[] at (7,-0.3)   { \tiny$2n$};
         \node[] at (4,-1.25)   { $M_1$};

\end{tikzpicture}
    \end{minipage}
     \begin{minipage}{.2\textwidth}
     \begin{tikzpicture}[scale=.47]
	
        \filldraw  (0,0) circle (2pt); 
        \filldraw  (6.5,0) circle (2pt); 
        \filldraw  (5,0) circle (2pt); 
        \filldraw  (3.5,0) circle (2pt); 
        
        \draw (0,0) arc (180:0:2.5) ;
        \draw[dotted][thick] (3.5,0) arc (180:0:1.5) ;
        \node[] at (0,-0.3)   { \tiny$i$};
        \node[] at (3.3,-0.3)   { \tiny $-2$};
         \node[] at (4.8,-0.3)   { \tiny $-1$};
         \node[] at (6.5,-0.3)   { \tiny $2n$};
        
         \node[] at (4,-1.25)   { $M_2$};
\end{tikzpicture}
\end{minipage}
\begin{minipage}{.2\textwidth}
     \begin{tikzpicture}[scale=.47]

      \filldraw  (0,0) circle (2pt); 
        \filldraw  (6,0) circle (2pt); 
        \filldraw  (5,0) circle (2pt); 
        \filldraw  (3,0) circle (2pt); 
        \draw (0,0) arc (180:0:3) ;
        \draw[dotted][thick] (3,0) arc (180:0:1) ;
        \node[] at (0,-0.3)   { \tiny$i$};
        \node[] at (2.8,-0.3)   { \tiny$-2$};
         \node[] at (4.8,-0.3)   { \tiny$-1$};
         \node[] at (6,-0.3)   { \tiny$2n$};
         \node[] at (3,-1.25)   { $M_3$};
        
     \end{tikzpicture}
\end{minipage}
\begin{minipage}{.25\textwidth}
     \begin{tikzpicture}[scale=.47]
      \filldraw  (0,0) circle (2pt); 
        \filldraw  (6,0) circle (2pt); 
        \filldraw  (5,0) circle (2pt); 
        \filldraw  (4,0) circle (2pt); 
        \filldraw  (7,0) circle (2pt);
        \filldraw  (2,0) circle (2pt); 
        \draw (0,0) arc (180:0:3.5) ;
        \draw (2,0) arc (180:0:1.5) ;
        \draw[dotted][thick] (4,0) arc (180:0:1) ;
        \node[] at (0,-0.25)   { \tiny$i$};
        \node[] at (2,-0.25)   { \tiny$j$};
        \node[] at (3.8,-0.3)   { \tiny$-3$};
         \node[] at (4.8,-0.3)   { \tiny$-2$};
         \node[] at (5.8,-0.3)   { \tiny$-1$};
         \node[] at (7,-0.25)   { \tiny$2n$};

         \node[] at (3.5,-1.25)   { $M_4$};
     \end{tikzpicture}
\end{minipage}
\begin{minipage}{.3\textwidth}
     \begin{tikzpicture}[scale=.5]
       \filldraw  (0,0) circle (2pt); 
        \filldraw  (4.5,0) circle (2pt); 
        \filldraw  (5,0) circle (2pt); 
        \filldraw  (2.5,0) circle (2pt); 
        \filldraw  (7,0) circle (2pt);
        \filldraw  (2,0) circle (2pt); 
        \draw (0,0) arc (180:0:2.5) ;
        \draw (2,0) arc (180:0:2.5) ;
        \draw[dotted][thick] (2.5,0) arc (180:0:1) ;
        \node[] at (0,-0.25)   { \tiny$i$};
        \node[] at (1.8,-0.25)   { \tiny$j$};
        \node[] at (2.3,-0.3)   { \tiny$-3$};
         \node[] at (4.3,-0.3)   { \tiny$-2$};
         \node[] at (4.9,-0.3)   { \tiny$-1$};
         \node[] at (7,-0.25)   { \tiny$2n$};

         \node[] at (3.5,-1.25)   { $M_5$};
     \end{tikzpicture}
\end{minipage}
\begin{minipage}{.32\textwidth}
\begin{tikzpicture}[scale=.5]
\filldraw  (0,0) circle (2pt); 
        \filldraw  (2.75,0) circle (2pt); 
        \filldraw  (4.7,0) circle (2pt); 
        \filldraw  (5.45,0) circle (2pt); 
        \filldraw  (6,0) circle (2pt); 
        \filldraw  (2,0) circle (2pt); 
        \filldraw  (7.5,0) circle (2pt);
        \filldraw  (1.5,0) circle (2pt); 
        \draw (0,0) arc (180:0:3) ;
        \draw (1.5,0) arc (180:0:3) ;
        \draw (2,0) arc (180:0:1.35) ;
        \draw[dotted][thick] (2.75,0) arc (180:0:1.35) ;
        \node[] at (0,-0.25)   { \tiny$i$};
        \node[] at (1.5,-0.25)   { \tiny$j$};
        \node[] at (2,-0.25)   { \tiny$k$};
        \node[] at (2.55,-0.3)   { \tiny$-4$};
        \node[] at (4.5,-0.3)   { \tiny$-3$};
         \node[] at (5.35,-0.3) { \tiny$-2$};
         \node[] at (5.9,-0.3)   { \tiny$-1$};
         \node[] at (7.5,-0.25) { \tiny$2n$};

         \node[] at (4,-1.25)   { $M_6$};
\end{tikzpicture}
\end{minipage}
\begin{minipage}{.28\textwidth}
     \begin{tikzpicture}[scale=.47]
      \filldraw  (-0.5,0) circle (2pt); 
        \filldraw  (8,0) circle (2pt); 
        \filldraw  (7,0) circle (2pt); 
        \filldraw  (3,0) circle (2pt); 
        \filldraw  (5.5,0) circle (2pt); 
        \filldraw  (4,0) circle (2pt);

        \draw (-0.5,0) arc (180:0:3) ;
        \draw  (3,0) arc (180:0:2.5) ;
        \draw[dotted][thick] (4,0) arc (180:0:1.5) ;
        \node[] at (-0.5,-0.25)   { \tiny$i$};
       
        \node[] at (2.8,-0.3)   {\tiny {$j$}};
        \node[] at (3.8,-0.3)   {\tiny {$-3$}};
        \node[] at (5.3,-0.3)   {\tiny {$-2$}};
         \node[] at (6.8,-0.3)   {\tiny $-1$};
         \node[] at (8,-0.25)   {\tiny $2n$};
        \node[] at (4,-1.25)   { $M_7$};
       
\end{tikzpicture}
\end{minipage}
\caption{Permitted perfect matchings}  
\label{Fig3}
 \end{figure}

Let $m_i(n)$ denote the number of perfect matchings of type $M_i$ on $2n$ vertices,
for $i=1,2,\ldots,7$, and let $a_n$ be the total number of $P_9$- avoiding perfect matchings
on $2n$ vertices.
Perfect matchings of types $M_1$, $M_2$, and $M_3$ can be obtained by adding a single arc
to any $P_9$- avoiding matching on $2(n-1)$ vertices, without imposing any additional
conditions. Hence, $m_1(n)=m_2(n)=m_3(n)=a_{n-1}.$

Perfect matchings of type $M_4$ on $2n$ vertices can be constructed from perfect matchings of type $M_3$, $M_4$ and $M_8$ on $2(n-1)$ vertices by adding a rightmost opening arc nested under the rightmost closer and crossing the second rightmost closer. Thus,
\begin{equation}\label{eq1}
  m_4(n)=m_3(n-1)+m_4(n-1)+m_7(n-1) =a_{n-2}+m_4(n-1)+m_7(n-1).
\end{equation}


Perfect matchings of type $M_5$ and $M_7$ on $2n$ vertices can be constructed from perfect matchings of type $M_2$, $M_5$, and $M_6$ on $2(n-1)$ vertices by adding a rightmost opening arc nested under a crossing. Hence,
$m_5(n)=m_7(n)=m_2(n-1)+m_5(n-1)+m_6(n-1)=a_{n-2}+m_5(n-1)+m_6(n-1).$
Since perfect matchings of types $M_1$, $M_3$, $M_4$ and $M_7$ cannot contribute to $M_5$ and $M_7$, we obtain
$m_5(n)=m_7(n)=a_{n-1}-m_1(n-1)-m_3(n-1)-m_4(n-1)-m_7(n-1).$
From \eqref{eq1} we get,
\begin{equation}\label{eq2}
   m_5(n)=m_7(n) = a_{n-1}-2a_{n-2}-(m_4(n-1)-a_{n-2})
   = a_{n-1}-a_{n-2}-m_4(n).
\end{equation}
Now combining \eqref{eq1} and \eqref{eq2} we get 
\begin{equation}\label{eq4}
    m_4(n)=2a_{n-2}-a_{n-3}.
\end{equation}
Perfect matchings of type $M_6$ on $2n$ vertices can be constructed from perfect matchings of type $M_5$ and $M_6$ on $2(n-1)$ vertices by adding a rightmost opening arc nested under a crossing and crossing the second rightmost closer. Therefore,
\begin{equation}\label{eq3}
    m_6(n)=m_5(n-1)+m_6(n-1)
    = m_5(n)-a_{n-2}
    = a_{n-1}-2a_{n-2}-m_4(n).
\end{equation}

Summing over all types and using \eqref{eq2}, \eqref{eq4} and \eqref{eq3}, we obtain
    $a_n=\sum_{i=1}^{7} m_i(n)=6a_{n-1}-8a_{n-2}+2a_{n-3}$,
completing the proof.
\end{proof}

\subsection{Class XI}

We consider $P_{20}$.
For an arc $A$, write $o(A)$ and $c(A)$ for its opening and closing
vertices. An arc $B$ is a \emph{right-crosser} of $A$ if
$o(A)<o(B)<c(A)<c(B)$, and a \emph{left-crosser} of $A$ if
$o(B)<o(A)<c(B)<c(A)$. We use the permutation-pattern convention in
which a permutation $\pi$ of $[3]$ corresponds to the matching
$\{(j,7-\pi(j)):1\leq j\leq3\}$. Thus every occurrence has all three
opening vertices before all three closing vertices. Standardization
means relabelling the selected vertices increasingly by consecutive
integers.
The following observation makes the sequence of successive
right-crossers unique.

\begin{lemma}\label{lem:P20-crossers}
In a perfect matching avoiding $P_{20}$, every arc has at most one
right-crosser and at most one left-crosser.
\end{lemma}

\begin{proof}
Suppose that $B$ and $C$ are right-crossers of an arc $A$, with
$o(B)<o(C)$. Then, $o(A)<o(B)<o(C)<c(A).$
If $c(B)<c(C)$, the arcs $A,B,C$ form $321$. If $c(C)<c(B)$, they
form $312$. Both possibilities contradict avoidance.
Now suppose that $B$ and $C$ are left-crossers of $A$, with
$o(B)<o(C)<o(A)$. If $c(B)<c(C)$, the arcs $B,C,A$ form $321$,
whereas if $c(C)<c(B)$, they form $231$. Both possibilities are
again forbidden.
\end{proof}
Let $M\in\mathcal{M}_n(P_{20})$ be nonempty. Starting with the arc
$A_1$ opening at vertex $1$, let $A_{i+1}$ be the right-crosser of
$A_i$ whenever it exists. The opening vertices strictly increase,
and Lemma~\ref{lem:P20-crossers} ensures uniqueness. We call the
resulting maximal sequence $A_1,\ldots,A_k$ the
\emph{right-crosser chain} of $M$.

The chain endpoints determine regions consisting of the vertices
strictly between consecutive chain endpoints, together with the
vertices after the last chain endpoint. For $1\leq i\leq k$, let
$L_i$ be the region immediately after $o(A_i)$ and $M_i$ the region
immediately after $c(A_i)$. In particular, $M_k$ contains the vertices
after $c(A_k)$.
The next lemma shows that these regions give a recursive decomposition
of the matching.

\begin{lemma}\label{lem:P20-decomposition}
Let $M\in\mathcal{M}_n(P_{20})$ be nonempty, with right-crosser chain
$A_1,\ldots,A_k$. Then
\[
    c(A_i)<o(A_{i+2})\qquad(1\leq i\leq k-2).
\]
Every arc outside the chain has both endpoints in a single region
$L_i$ or $M_i$. Therefore, these regions induce perfect matchings
that avoid $P_{20}$ after standardization.
\end{lemma}

\begin{proof}
By the definition of a right-crosser,
$o(A_i)<o(A_{i+1})<c(A_i)<c(A_{i+1})   \>\> (1\leq i<k).$
If $o(A_{i+2})<c(A_i)$, then $A_{i+1}$ and $A_{i+2}$ would both be
right-crossers of $A_i$, contrary to Lemma~\ref{lem:P20-crossers}.
This proves the first assertion and determines the order of all
chain endpoints.

Suppose that an arc $B$ opens in $L_1$. When $k=1$, an endpoint
$c(B)>c(A_1)$ would make $B$ a right-crosser of $A_1$, contrary to
maximality of the chain. Thus, $B$ closes in $L_1$. When $k\geq2$,
if $o(A_2)<c(B)<c(A_1)$, the arcs $A_1,B,A_2$ form $231$. If
$c(B)>c(A_1)$, then $B$ and $A_2$ are both right-crossers of $A_1$.
Neither is possible, so $B$ again closes in $L_1$.

For $2\leq i\leq k$, the region $L_i$ lies between $o(A_i)$ and
$c(A_{i-1})$. If an arc $B$ opening in $L_i$ closed after
$c(A_{i-1})$, then $A_i$ and $B$ would both be right-crossers of
$A_{i-1}$. Therefore, $B$ closes in $L_i$.
For $1\leq i\leq k-2$, the region $M_i$ lies between $c(A_i)$ and
$o(A_{i+2})$. Suppose that an arc $B$ opens in $M_i$ but closes
after $o(A_{i+2})$. If $c(B)<c(A_{i+1})$, the arcs
$A_{i+1},B,A_{i+2}$ form $231$. If $c(B)>c(A_{i+1})$, then $B$
and $A_{i+2}$ are both right-crossers of $A_{i+1}$. Both alternatives
are impossible, so $B$ closes in $M_i$.

When $k\geq2$, the region $M_{k-1}$ lies between $c(A_{k-1})$ and
$c(A_k)$. An arc opening there and closing after $c(A_k)$ would be
a right-crosser of $A_k$, contrary to maximality. Finally, an arc
opening in $M_k$ necessarily also closes there.
Every arc outside the chain opens in one of these regions, and the
preceding arguments show that it closes in the same region. Hence
each region induces a perfect matching. Its standardization avoids
$P_{20}$ because pattern avoidance is inherited by submatchings.
\end{proof}

We also use $L_i$ and $M_i$ for the standardized submatchings supplied
by Lemma~\ref{lem:P20-decomposition}. For $k\geq3$, their positions
are described by the endpoint word
\begin{equation}\label{eq:P20-endpoint-decomposition}
    o(A_1)L_1\;o(A_2)L_2c(A_1)M_1\;\cdots\;
    o(A_k)L_kc(A_{k-1})M_{k-1}\;c(A_k)M_k.
\end{equation}
The block $o(A_i)L_ic(A_{i-1})M_{i-1}$ occurs once for each
$i=2,\ldots,k$, and each $L_i$ or $M_i$ denotes the complete block
of vertices of that submatching. For $k=2$, the word is
$o(A_1)L_1\;o(A_2)L_2c(A_1)M_1\;c(A_2)M_2,$
and for $k=1$ it is
    $o(A_1)L_1c(A_1)M_1.$
To reconstruct a matching from a tree, we need the following
property.

\begin{lemma}\label{lem:P20-insertion}
Let $N$ and $N'$ be perfect matchings avoiding $P_{20}$. Inserting
all vertices of $N'$ as a consecutive block into any gap of $N$
preserves avoidance. This includes the gaps before the first vertex
and after the last vertex. If $N$ is empty, the resulting matching
is $N'$.
\end{lemma}

\begin{proof}
The assertion is immediate if either matching is empty. Otherwise,
suppose that the resulting matching contains a forbidden occurrence.
It must use arcs from both $N$ and $N'$, since their internal endpoint
orders are unchanged.
Every selected arc of $N$ must span the entire inserted block.
Indeed, an arc of $N$ either spans that block or lies entirely before
or after it. In either of the latter cases, a selected closing
vertex would precede a selected opening vertex, which is impossible
in a permutation-pattern occurrence.

If the occurrence uses one arc from $N'$ and two from $N$, the arc
from $N'$ opens last and closes first. The pattern is therefore
$123$ or $213$. If it uses two arcs from $N'$ and one from $N$, the
arc from $N$ opens first and closes last, so the pattern is $123$
or $132$. None belongs to $P_{20}$, a contradiction.
\end{proof}

\begin{theorem}\label{thm: FC1}
Let $P_{20}=\{231,312,321\}$. For every $n\geq0$, the number of perfect
matchings on $[2n]$ that avoid $P_{20}$ is
\[
    \bigl|\mathcal{M}_n(P_{20})\bigr|
    =C_n^3
    =\frac{1}{3n+1}\binom{3n+1}{n}.
\]
\end{theorem}
\begin{proof}
Let $\mathcal{T}_n$ be the set of rooted ordered ternary trees with
$n$ vertices, where every vertex has three ordered, possibly empty,
subtrees, called its left, middle, and right subtrees. We include
the empty tree in $\mathcal{T}_0$. The \emph{right spine} of a nonempty
tree is the maximal sequence of vertices obtained by following
right children from the root.
We define maps $\phi$ from $P_{20}$-avoiding matchings to ternary
trees and $\psi$ in the opposite direction, recursively on the
number of arcs or vertices. Both maps send the empty object to the
empty object.

Let $M\in\mathcal{M}_n(P_{20})$ be nonempty, with right-crosser
chain $A_1,\ldots,A_k$. In $\phi(M)$, the corresponding vertices
$r_1,\ldots,r_k$ form the right spine, with $r_1$ as the root.
For each $i$, the left subtree of $r_i$ is $\phi(L_i)$ and its
middle subtree is $\phi(M_i)$. By
Lemma~\ref{lem:P20-decomposition}, these submatchings avoid $P_{20}$
and have fewer than $n$ arcs, so the recursive construction is
well-defined. Every arc contributes one vertex, and hence
$\phi(M)\in\mathcal{T}_n$. Figure~\ref{Fig10} illustrates the map.

Conversely, let $T\in\mathcal{T}_n$ be nonempty, with right spine
$r_1,\ldots,r_k$. For each $i$, let $L_i$ and $M_i$ be the matchings
obtained recursively from the left and middle subtrees of $r_i$.
The matching $\psi(T)$ has arcs $A_1,\ldots,A_k$ whose endpoints
and intervening submatchings occur in the order
\eqref{eq:P20-endpoint-decomposition}, with the displayed shorter
words when $k=1$ or $k=2$. All vertices are labelled increasingly,
and each $o(A_i)$ is matched to $c(A_i)$.
The matching consisting only of the chain arcs avoids $P_{20}$.
Indeed, for any three such arcs with indices $i<j<\ell$, we have
$\ell\geq i+2$ and $c(A_i)<o(A_\ell)$. These arcs therefore cannot
form a permutation pattern. By induction, the reconstructed
submatchings $L_i$ and $M_i$ avoid $P_{20}$, and
Lemma~\ref{lem:P20-insertion} shows that inserting them in the
prescribed gaps preserves avoidance. This includes the final block
$M_k$. Thus $\psi(T)\in\mathcal{M}_n(P_{20})$.

No inserted arc crosses a chain arc, since both of its endpoints
lie in a single gap between chain endpoints or after the last
endpoint. Consequently, the right-crosser chain of $\psi(T)$ is
exactly $A_1,\ldots,A_k$, and its complementary submatchings are
exactly the reconstructed $L_i$ and $M_i$. The uniqueness of the
decomposition and induction on $n$ now give
$psi(\phi(M))=M$ and $\phi(\psi(T))=T.$
Hence $\phi$ is a bijection.
For completeness, let
$T(z)=\sum_{n\geq0}|\mathcal{T}_n|z^n$. The three ordered subtrees
at the root give $T(z)=1+zT(z)^3$. Setting $U(z)=T(z)-1$, we obtain
$U(z)=z(1+U(z))^3$. Lagrange inversion gives, for $n\geq1$,
$\displaystyle
    |\mathcal{T}_n|
    =\frac{1}{n}[u^{n-1}](1+u)^{3n}
    =\frac{1}{n}\binom{3n}{n-1}
    =\frac{1}{3n+1}\binom{3n+1}{n}.
$
The empty matching and empty tree give the same formula when $n=0$.
\end{proof}

\begin{figure}[!htbp]
\centering
\begin{minipage}{.58\textwidth}
\centering
\begin{tikzpicture}[scale=.45]
    \filldraw (0,0) circle (2pt);
    \filldraw (6,0) circle (2pt);
    \filldraw (8,0) circle (2pt);
    \filldraw (12,0) circle (2pt);
    \filldraw (14,0) circle (2pt);
    \filldraw (20,0) circle (2pt);

    \draw (0,0) arc (180:0:4);
    \draw (6,0) arc (180:0:4);
    \draw (12,0) arc (180:0:4);

    \draw (3.5,1) ellipse (1.5cm and 1cm);
    \draw (7,1) ellipse (.73cm and .73cm);
    \draw (10,1) ellipse (1.5cm and 1cm);
    \draw (13,1) ellipse (.73cm and .73cm);
    \draw (17,1) ellipse (1.5cm and 1cm);

    \node at (3.5,1) {\scriptsize $L_1$};
    \node at (7,1) {\scriptsize $L_2$};
    \node at (10,1) {\scriptsize $M_1$};
    \node at (13,1) {\scriptsize $L_3$};
    \node at (17,1) {\scriptsize $M_2$};

    \node at (4,4.5) {\scriptsize $A_1$};
    \node at (10,4.5) {\scriptsize $A_2$};
    \node at (16.5,4.5) {\scriptsize $A_3$};
\end{tikzpicture}
\end{minipage}
\begin{minipage}{.08\textwidth}
\centering
$\longrightarrow$
\end{minipage}
\begin{minipage}{.28\textwidth}
\centering
\begin{tikzpicture}[scale=.55]
    \filldraw (0,.75) circle (2pt);        

    \filldraw (-1.5,-1) circle (2pt);      
    \filldraw (0,-1) circle (2pt);         
    \filldraw (1.5,-1) circle (2pt);       

    \filldraw (0,-2.75) circle (2pt);      
    \filldraw (1.5,-2.75) circle (2pt);    
    \filldraw (3,-2.75) circle (2pt);      

    \filldraw (1.5,-4.5) circle (2pt);     

    \draw (0,.75) -- (-1.5,-1);
    \draw (0,.75) -- (0,-1);
    \draw (0,.75) -- (1.5,-1);

    \draw (1.5,-1) -- (0,-2.75);
    \draw (1.5,-1) -- (1.5,-2.75);
    \draw (1.5,-1) -- (3,-2.75);

    \draw (3,-2.75) -- (1.5,-4.5);

    \node at (0,1.2) {\scriptsize $A_1$};
    \node at (-1.5,-1.4) {\scriptsize $L_1$};
    \node at (0,-1.4) {\scriptsize $M_1$};
    \node at (2.1,-1) {\scriptsize $A_2$};

    \node at (0,-3.15) {\scriptsize $L_2$};
    \node at (1.5,-3.15) {\scriptsize $M_2$};
    \node at (3.6,-2.75) {\scriptsize $A_3$};

    \node at (1.5,-4.9) {\scriptsize $L_3$};
\end{tikzpicture}
\end{minipage}
\caption{\small The right-crosser chain becomes the right spine of the ternary
tree. On the tree, the labels $L_i$ and $M_i$ represent the recursively
constructed left and middle subtrees at the vertex corresponding to
$A_i$, rather than individual vertices. The terminal region $M_3$ is
empty in this example.}
\label{Fig10}
\end{figure}

\subsection{Class III}
\begin{theorem}\label{thm: FC2}
    Let $P_{18}=\{213,231,321\}$. Then the number of perfect matchings on $[2n]$ that avoid $P_{18}$ is
    given by the Fuss--Catalan number
    \(
    C_n^{3}=\frac{1}{3n+1}\binom{3n+1}{n}.
    \)
\end{theorem}

\begin{proof}
For a matching $M$ on $[2n]$, label its arcs
$1,2,\ldots,n$ in increasing order of their opening vertices, and
let $w(M)$ be the resulting sequential form of the matching. Thus each
label occurs twice, and the first occurrence of $i$ precedes the
first occurrence of $i+1$. If $u$ is a word in which every letter
occurs twice, let $\operatorname{can}(u)$ denote the word obtained
by relabelling its letters in order of first occurrence.
Define the \emph{reflection} of $M$ by
$\operatorname{rev}(M)
=
\bigl\{
(2n+1-j,\,2n+1-i):(i,j)\in M
\bigr\}.$
Thus $\operatorname{rev}(M)$ is obtained by reflecting the linear
representation of $M$ in a vertical line. 
Clearly,
$\operatorname{rev}$ is an involution. Moreover, if
$w(M)=w_1w_2\cdots w_{2n},$ then
\begin{equation}\label{eq:reflected-word}
w\bigl(\operatorname{rev}(M)\bigr)
=
\operatorname{can}\bigl(w_{2n}w_{2n-1}\cdots w_1\bigr).
\end{equation}
Reflection carries every occurrence of a matching
pattern $\tau$ to an occurrence of the reflected pattern
$\operatorname{rev}(\tau)$, and vice versa.
The relevant length-three patterns have the following matching
words:
\[
\begin{array}{c|c|c|||c|c|c|||c|c|c}
\tau & w(\tau) & w(\operatorname{rev}(\tau)) & 
\tau & w(\tau) & w(\operatorname{rev}(\tau)) &
\tau & w(\tau) & w(\operatorname{rev}(\tau))\\
\hline
213 & 123312 & 123312 & 
231 & 123213 & 123132 & 
321 & 123123 & 123123
\end{array}
\]
Since $123312$, $123132$, and $123123$ are respectively the
matching words of $213$, $312$, and $321$, it follows that
$\operatorname{rev}(213)=213,
\operatorname{rev}(231)=312,
\operatorname{rev}(321)=321.$  
Thus, reflection restricts to a bijection
\begin{equation}\label{eq:reflection-bijection}
\operatorname{rev}:
\mathcal{M}_n(213,231,321)
\longrightarrow
\mathcal{M}_n(213,312,321).
\end{equation}

Figure~\ref{Fig11} illustrates the map. The matching on the left has
matching word $123132$, and hence is the pattern $312$; its
reflection has matching word $123213$, and hence is the pattern
$231$.

\begin{figure}[H]
\centering
\begin{tikzpicture}[x=.72cm,y=.72cm,>=stealth]
    \foreach \x/\lab in {0/1,1/2,2/3,3/4,4/5,5/6}{
        \filldraw (\x,0) circle (1.7pt);
        \node[below=3pt] at (\x,0) {\scriptsize \lab};
    }

    \draw (0,0)
        .. controls (0,1.45) and (3,1.45) .. (3,0);
    \draw (1,0)
        .. controls (1,2.25) and (5,2.25) .. (5,0);
    \draw (2,0)
        .. controls (2,1.05) and (4,1.05) .. (4,0);

    \node at (2.5,2.7) {$M$};
    \node at (2.5,-1)
        {\scriptsize $w(M)=123132$};

    \draw[->,thick] (6.25,1.1) --
        node[above] {$\operatorname{rev}$} (7.75,1.1);

    \foreach \x/\lab in {9/1,10/2,11/3,12/4,13/5,14/6}{
        \filldraw (\x,0) circle (1.7pt);
        \node[below=3pt] at (\x,0) {\scriptsize \lab};
    }

    \draw (9,0)
        .. controls (9,2.25) and (13,2.25) .. (13,0);
    \draw (10,0)
        .. controls (10,1.05) and (12,1.05) .. (12,0);
    \draw (11,0)
        .. controls (11,1.45) and (14,1.45) .. (14,0);

    \node at (11.5,2.7)
        {$\operatorname{rev}(M)$};
    \node at (11.5,-1)
        {\scriptsize
        $w(\operatorname{rev}(M))=123213$};
\end{tikzpicture}
\caption{\small An example of the reflection bijection. Here
$M=\{(1,4),(2,6),(3,5)\}$ and
$\operatorname{rev}(M)=\{(1,5),(2,4),(3,6)\}$.}
\label{Fig11}
\end{figure}

We identify the class on the right-hand side of
\eqref{eq:reflection-bijection}. An occurrence of the partial
pattern $12312$ selects five endpoints belonging to three arcs.
Restricting to those three arcs and restoring the omitted endpoint
of the third arc gives one of the three complete matching words
$123312,
123132,
123123.$
These are precisely the matching patterns $213$, $312$, and $321$.
Additionally, each of the above three words 
contains $12312$ as a subsequence. Hence
\begin{equation}\label{eq:partial-complete-equivalence}
\mathcal{M}_n(12312)
=
\mathcal{M}_n(213,312,321).
\end{equation}

Combining \eqref{eq:reflection-bijection} and
\eqref{eq:partial-complete-equivalence}, reflection gives the
explicit bijection 
$\operatorname{rev}:
\mathcal{M}_n(P_{18})
\longrightarrow
\mathcal{M}_n(12312).$
By \cite[Theorem~2.3]{ChenMansourYan2006}, the number of
$12312$-avoiding perfect matchings on $[2n]$ is the $3$-Fuss Catalan number
$\displaystyle
\left|\mathcal{M}_n(12312)\right|
=
\frac{1}{2n+1}\binom{3n}{n} = \frac{1}{3n+1}\binom{3n+1}{n},
$ completing the proof.
\end{proof}

\subsection{Class I}
\begin{theorem}\label{class1}
    For $P_{16}=\{132,312,321\}$,  the generating function of $a_n$ is $A(x)= \frac{3\sqrt{1 - 4x} - 1}{2\sqrt{1 - 4x} - 4x}.$
\end{theorem}
\begin{proof}
We determine the generating function using the Catalan generating function. Since any disconnected matching is a disjoint union of connected components, it is enough to first count connected perfect matchings.

Because the perfect matchings avoid the patterns $132$, $312$, and $321$, Figure \ref{Fig1} implies that they cannot contain a crossing below a nesting (corresponding to $132$), three mutually crossing arcs (corresponding to $321$), or two nested arcs together with a right crossing (corresponding to $312$). In particular, avoidance of $321$ implies that along the boundary of any connected component there are exactly $k$ boundary arcs such that only two consecutive arcs may cross; see Figure \ref{Fig8}.
\begin{figure}[!htbp]
\begin{minipage}{.6\textwidth}
\begin{tikzpicture}[baseline=(current bounding box),scale=.48]
        \filldraw  (0,0) circle (2pt); 
        \filldraw  (7,0) circle (2pt); 
        \filldraw  (5,0) circle (2pt); 
        \filldraw  (3.5,0) circle (2pt); 
        \filldraw  (8.5,0) circle (2pt); 
        \filldraw  (11.5,0) circle (2pt); 
        \filldraw  (13,0) circle (2pt);
        \filldraw  (16.5,0) circle (2pt);
        
        \draw (0,0) arc (180:0:2.5) ;
        \draw(3.5,0) arc (180:0:2.5) ;
         \draw[dotted][thick](7,0) arc (180:100:2.5) ;
         \draw[dotted][thick](13,0) arc (0:80:2.5) ;
         \draw(11.5,0) arc (180:0:2.5) ;
         \node[] at (10.1,0)   { \Huge$\cdots$};
        
         \node[] at (0,-.9)   { \textcolor{white}{0}};
\end{tikzpicture}
\caption{$k$ crossing arcs}  
\label{Fig8}
\end{minipage}
\begin{minipage}{.35\textwidth}
\begin{tikzpicture}[baseline=(current bounding box),scale=.43]

 \filldraw  (0,0) circle (2pt); 
        \filldraw  (8.5,0) circle (2pt); 
        \filldraw  (5,0) circle (2pt); 
        \filldraw  (3.5,0) circle (2pt); 
        
        \draw (0,0) arc (180:0:2.5) ;
        \draw[dotted][thick] (3.5,0) arc (180:0:2.5) ;
        \draw [decorate,decoration = {calligraphic brace,mirror}][thick] (0,-0.7) --  (5,-0.7);
        \draw [decorate,decoration = {calligraphic brace,mirror}][thick] (3.5,-1) --  (8.5,-1);
        
        \node[] at (-.1,-0.4)   { \scriptsize$A$};
        \node[] at (3.5,-0.4)   { \scriptsize $C$};
         \node[] at (4.9,-0.4)   {\scriptsize $B$};
         \node[] at (8.4,-0.4)   { \scriptsize $D$};
         \node[] at (2.5,0.5){ $d_n$};
         \node[] at (2.5,-1.2){ $g_n$};
         \node[] at (6,-1.5){ $d_n$};

\end{tikzpicture}
\caption{\small{Crossing arcs}  }
\label{Fig9}
\end{minipage}
\end{figure}

First consider a single boundary arc. Inside such an arc, only noncrossing arcs may occur. Since the number of noncrossing perfect matchings on $2n$ vertices is the Catalan number $C_n$, the corresponding generating function is
$D(x)=xC(x),$ where $C(x)$ is the Catalan generating function. Let
$D(x)=\sum_{n\geq 0} d_nx^n.$

Next consider two consecutive boundary arcs crossing each other, as shown in Figure \ref{Fig9}. Suppose the two regions with endpoints $A,B$ and $C,D$ each contain $2n$ vertices. As before, the arc $AB$ can contain only noncrossing arcs, contributing $d_n$. The left endpoint $C$ of the second arc can then be chosen in $(2n-1)$ possible positions. Hence the number of such configurations is $(2n-1)d_n$. Therefore,
$$
G(x)=\sum_{n\geq 0}g_nx^n=\sum_{n\geq 0}(2n-1)d_nx^n
    =2xD'(x)-D(x).
$$

Since the matching must also avoid $312$, no arc nested under $CD$ may cross an arc inside $AB$. Thus the arc $CD$ also contains only noncrossing arcs disjoint from those inside $AB$, again contributing $d_n$.
More generally, if we have $k$ boundary arcs such that only consecutive pairs cross, then the first $(k-1)$ arcs contribute $G(x)^{k-1}$, while the last arc contributes $D(x)$. Hence the generating function for all connected perfect matchings is
$\displaystyle 
\sum_{k\geq 1} G(x)^{k-1}D(x)
=
\frac{D(x)}{1-G(x)}.
$

Finally, since an arbitrary perfect matchings is a disjoint union of connected components, the generating function for all perfect matchings avoiding $P_{16}$ is
$\displaystyle A(x)=\frac{1}{1-\frac{D(x)}{1-G(x)}}.$
Now we calculate the exact closed form of $A(x)$. 
Clearly,$\displaystyle D'(x) = \frac{1}{\sqrt{1 - 4x}}$.
We further get
\begin{equation}\label{eq14}
A(x) = \frac{1}{1 - \frac{D(x)}{1 - 2xD'(x) + D(x)}}.
\end{equation}
Let $y = \sqrt{1 - 4x}$. 
Thus, 
    $\displaystyle D(y) = \frac{1 - y}{2}  \mbox{ and }
    2xD'(x) = \frac{2}{y} \left(\frac{1 - y^2}{4}\right).$
We next evaluate the denominator of the inner fraction:
$1 - 2xD'(x) + D(x) = 1 - \frac{1 - y^2}{2y} + \frac{1 - y}{2} = \frac{3y - 1}{2y}.  \mbox{Thus,} $

$$\frac{D(x)}{1-G(x)}= \frac{D(x)}{1 - 2xD'(x) + D(x)} = \frac{\frac{1 - y}{2}}{\frac{3y - 1}{2y}} = \frac{y(1 - y)}{3y - 1}.$$
Plugging this in \eqref{eq14} and simplifying yields 
$\displaystyle A(y)=\frac{3y - 1}{y^2 + 2y - 1}.$
By substituting $y = \sqrt{1 - 4x}$ 
we get $\displaystyle A(x) 
    =\frac{3\sqrt{1 - 4x} - 1}{2\sqrt{1 - 4x} - 4x},$ completing the proof.
\end{proof}

\section{Sets of Permutations Arising as Transversals}
In this section, we consider the permutations that can occur as transversals of Ferrers boards avoiding a given set of patterns. More precisely, for a pattern set $P$, we ask which permutations can be realized as $P$-avoiding transversals of some Ferrers board. We determine these realizable permutations for the pattern sets considered below and describe the resulting permutation classes.

\begin{theorem}\label{set}
    The set of permutations that arise as transversals avoiding a given pattern set $P$ on some Ferrers board is precisely the set $S$. The sets $S$ and the corresponding pattern sets $P$ are described in Table \ref{set-table}.

\begin{table}[!htbp]
    \centering
 {\renewcommand{\arraystretch}{1.5}   
\footnotesize{
\begin{tabular}{|l|l|l|l|l|l|}
\noalign{\hrule height 1pt}
\textbf{No.} & ~~~~~~~~~~~~~~~~~~~~~~~~\textbf{Set of patterns} $\mathbf{P}$ & ~~~~~~~~~~~~~~~~$\mathbf{S}$ & \textbf{OEIS}\\
 \noalign{\hrule height 1pt}
1  & $\{123,132,231\}$,$\{123,132,312\}$, $\{123,231,312\}$ & ~~~~~~~~~~$\SSS_n(123)$ &\href{https://oeis.org/search?q=A000108&go=Search}{A000108}\\
 \hline
2 & $\{132,213,231\}$, $\{132,213,312\}$, $\{132,213,321\}$ & ~~~~~~~~~~$\SSS_n(213)$ & \href{https://oeis.org/search?q=A000108&go=Search}{A000108}\\

& $\{213,231,312\}$, $\{213,231,321\}$, $\{213,312,321\}$& &\\
 \hline

3& ~~~~~~~~~~~~~~~~~~~~~~~~~~ $\{231,312,321\}$ & $\SSS_n(3214,2314,3124)$& \href{https://oeis.org/A106228}{A106228}\\
 \hline
4 & $\{123,132,213\}$, $\{123,213,231\}$, $\{123,213,312\}$ &~~~~~~~$\SSS_n(123,213)$& \href{https://oeis.org/A000079}{A000079}\\
&  ~~~~~~~~~~~~~~~~~~~~~~~~~~~$\{123,213,321\}$ & &\\
 \hline
5 & ~~~~~~~~~~~~~~~~~~~~~~~~~~~$\{132,231,321\}$ & $\SSS_n(1324,2314,3214)$ &\href{https://oeis.org/A026737} {A026737}\\
 \hline
6 & ~~~~~~~~~~~~~~~~~~~~~~~~~~~$\{132,312,321\}$ &  $\SSS_n(1324,3124,3214)$ & \href{https://oeis.org/A026737} {A026737}\\
 \hline
7 & $\{123,132,321\}$, $\{123,231,321\}$, $\{123,312,321\}$ & ~~~~~$\SSS_n(123, 3214)$ & \href{https://oeis.org/A001519} {A001519}\\
 \hline
8 & ~~~~~~~~~~~~~~~~~~~~~~~~~~~$\{132,231,312\}$& $\SSS_n(1324,2314,3124)$ & \href{https://oeis.org/A257562} {A257562}\\
 
\noalign{\hrule height 1pt}

\end{tabular}
}}
\caption{\small Pattern sets and the corresponding permutations arising as transversals of Ferrers boards }
\label{set-table}
\end{table} 
\end{theorem}

\begin{proof}
We only prove item (1); the remaining cases follow by analogous arguments.

\noindent
{\bf  Proof of (1), $ \mathbf{P= \{123,132,231\}, \{123,132,312\},\{123,231,312\}}$:}
Suppose that a permutation $\pi$ appears as a transversal of some Ferrers board 
while avoiding one of the above pattern sets. We claim that it avoids the pattern $123$. If not, then there clearly is an occurrence of $123$, say in cells $(r_1,c_1), (r_2,c_2), (r_3, c_3)$. Therefore, the board would contain all cells $(r_i,c_j)$ where $1\le i, j\le 3$, forming an occurrence of $123$ in the Ferrers board. Hence, $\pi \in \SSS_n(123)$.

Conversely, let $\pi \in \SSS_n(123)$. We show that there exists a Ferrers board 
$F$ such that $\pi$ appears as a transversal of $F$ avoiding one of the pattern sets 
$\{123,231,312\}$, $\{123,132,231\}$, or $\{123,132,312\}$.
We claim that the minimal board $F$ (a board such that no subboard contains $\pi$ as a transversal) that contains $\pi$ as a transversal does not contain occurrences of the aforementioned patterns. Suppose $\pi$ contains an occurrence of $231, 312$ or $132$ in $F$, then we argue as follows. Let $r_1<r_2<r_3$ and $c_1<c_2<c_3$ be the rows and columns of $F$ that form the occurrence of $231, 312$ or $132$. Then, we form the new board $F'$ by removing all cells including $(r_3, c_3)$ that lie above and to the right of $(r_3,c_3)$. Clearly, $F'$ is a subboard of $F$. We will just have to show that $F'$ contains $\pi$ as a transversal.

If $\pi$ avoids $231$, $312$, and $132$, then $\pi$ avoids each of the above pattern 
sets in the square Ferrers board, and we are done.
Otherwise, suppose that $\pi$ contains the pattern $231$. (The arguments for $312$ 
and $132$ are analogous.) Then there exist indices $i<j<k$ such that 
$\pi_k < \pi_i < \pi_j$, and these entries lie in a rectangular region of some 
Ferrers board containing $\pi$ as a transversal. In particular, this configuration 
determines a $3\times 3$ square subboard with columns $i,j,k$ (from left to right) 
and rows $\pi_k,\pi_i,\pi_j$ (from bottom to top).

Since $\pi$ avoids $123$, there is no entry of $\pi$ in the cell located at the 
upper-right corner of this $3\times 3$ subboard. Hence we may delete the cell in 
column $k$ and row $\pi_j$, obtaining a new Ferrers board $F'$. In $F'$, the 
configuration forming $231$ is no longer contained in a valid rectangle, and thus $\pi$ avoids $231$ as a transversal of $F'$. The deletion procedure is illustrated in Figure \ref{del}.

\begin{figure}[H]
\centering
{
\begin{minipage}{.24\textwidth}

\begin{tikzpicture}[scale=.5]
    
    \draw (0,0) grid (1,3);
    \draw (1,0) grid (2,3);
    \draw (2,0) grid (3,3);
    
    \filldraw   (0.5,1.5) circle (3pt); 
    \filldraw   (1.5,2.5) circle (3pt); 
    \filldraw  (2.5,.5) circle (3pt); 

    \node[] at (0.5,-0.5)   {\small $i$};
    \node[] at (1.5,-0.5)   {\small $j$};
    \node[] at (2.5,-0.5)   {\small $k$};

    \node[] at (-0.5,.5)   {\small $\pi_k$};
    \node[] at (-0.5,1.5)   {\small $\pi_i$};
    \node[] at (-0.5,2.5)   {\small $\pi_j$};
\end{tikzpicture}
\end{minipage}
\begin{minipage}{.08\textwidth}
   $\longrightarrow$
\end{minipage}
\begin{minipage}{.2\textwidth}
\begin{tikzpicture}[scale=.5]
    
    \draw (0,0) grid (1,3);
    \draw (1,0) grid (2,3);
    \draw (2,0) grid (3,2);
    
    \filldraw   (0.5,1.5) circle (3pt); 
    \filldraw   (1.5,2.5) circle (3pt); 
    \filldraw  (2.5,.5) circle (3pt); 

     \node[] at (0.5,-0.5)   { \small$i$};
    \node[] at (1.5,-0.5)   {\small $j$};
    \node[] at (2.5,-0.5)   {\small $k$};

    \node[] at (-0.5,.5) { \small$\pi_k$};
    \node[] at (-0.5,1.5){\small $\pi_i$};
    \node[] at (-0.5,2.5)   {\small $\pi_j$};

\end{tikzpicture}
\end{minipage}
}
\caption{Cell deletion }  
\label{del}
 \end{figure}

Repeating this process if necessary eliminates all occurrences of $231$, and the same argument applies to $312$ and $132$. Because the board is finite and we only delete empty cells (since $\pi$ avoids 123), this process strictly decreases the number of cells without deleting any elements of $\pi$, and thus must terminate. Therefore, there exists a Ferrers board in which $\pi$ appears as a transversal avoiding one of the given pattern sets.
\end{proof}
\begin{remark}
The proofs for the remaining cases in Theorem \ref{set} follow a similar constructive cell-deletion procedure. For cases where the resulting permutation class avoids length-$4$ patterns (such as Cases 3, 5, 6, 7, and 8), the deletion logic operates analogously on $4 \times 4$ bounding boxes rather than $3 \times 3$ subboards. In every case, the underlying avoidance conditions on $\pi$ guarantee that the critical corner cells of these bounding boxes are free of transversal elements. Because we only delete empty cells and the regions above and to the right of them, the operation strictly preserves both the transversal $\pi$ and the monotonic step-boundary required of a Ferrers board. As the initial $n \times n$ square board is finite, this iterative reduction strictly decreases the number of cells without affecting the transversal, ensuring the process terminates at a valid Ferrers board avoiding the specified pattern sets.
\end{remark}

\section{AI assisted results}
\label{sec:ai_assist}
Recall that $C_k=\frac{1}{k+1}\binom{2k}{k}$ is the $k$th Catalan number, and define 
the Fibonacci numbers $F_n$ for $n \geq 0$ by 
$F_0=0,\> F_1=1,\> F_{k+2}=F_{k+1}+F_k\> (k\geq0).$

\begin{theorem}\label{thm:matching-fibonacci-catalan}
Let $a_n$ be the number of perfect matchings on $[2n]$ avoiding
$P_1=\{123,132,213\}$. 
Then $a_0=a_1=1$, and when $n\geq2$, 
$\displaystyle a_n = a_{n-1} +
\sum_{k=0}^{n-2} F_{k+3}C_k\,a_{n-1-k}.$
\end{theorem}

\begin{proof}
An arc is \emph{open} if its opening endpoint has been encountered but its
closing endpoint has not. At each stage, order the open arcs by
their opening endpoints.  Only two things can happen.  
\begin{enumerate}
    \item Only the earliest-opened or second-earliest-opened arc
    may close. Otherwise one of $123,213$ will occur. 

    \item If the second-earliest-opened arc closes while at least
    three arcs are open, then the next closing endpoint must close
    the arc that was earliest-opened at that stage to avoid $132$. Opening
    endpoints may occur between these two closing endpoints.
\end{enumerate}

Conversely, suppose that a matching satisfies both closing rules.
An occurrence of $123$ or $213$ would require its third-opened arc
to close while the other two arcs remain open, contradicting the
first rule. In an occurrence of $132$, the middle-opened arc closes
first, followed by the last-opened arc and then the first-opened
arc. At the first of these closures, the first rule forces the
middle-opened arc to be second-earliest among all open arcs.
The second rule then forces the first-opened arc to close next,
contradicting the required closing order. Thus the two closing
rules characterize the avoiders.

Replace each opening endpoint by an up-step $U=(1,1)$ and each
closing endpoint by a down-step $D=(1,-1)$. This gives a Dyck path.
Also, a Dyck path together with the choice of an open arc
at every down-step determines a unique matching. The height of
the path is the number of open arcs, and the matching is connected
precisely when its path stays strictly above height zero between
its first and last vertices.
Call a closing condition \emph{pending} when the second rule
requires the next closure to be that of the earliest-opened arc.
Opening steps do not change whether such a condition is pending.
A pending condition is discharged by the required closure.
At every visit to height one, no closing condition is pending.
Indeed, a pending condition is created by a closure from height
at least three, leaving at least two arcs open; the next closure
discharges it.

Consider an path from height one back to height one that
stays strictly above height one internally. It has the form
$U P D,$
where $P$ is a Dyck path translated to start at height two.
Suppose that $P$ has semilength $m$. Its $m$ down-steps all start
at height at least three, whereas the final down-step of the
path goes from height two to height one.
Let $w_m$ be the number of admissible choices of closing arcs for
this path, starting with no pending condition. The closing
rules show that this number depends only on $m$, not on the shape
of $P$: opening steps impose no additional condition, all the
down-steps inside $P$ obey the same two-state rule, and the final
down-step starts at height two.

If $m=0$, either of the two open arcs may close at the final
down-step, so $w_0=2$. If $m=1$, closing the earliest-opened arc at
the down-step inside $P$ leaves two choices for the final closure,
whereas closing the second-earliest-opened arc leaves one forced
choice. Hence $w_1=3$.
For $m\geq2$, consider the first down-step inside $P$. Closing the
earliest-opened arc leaves $m-1$ such down-steps with no pending
condition. Closing the second-earliest-opened arc forces the next
down-step to close the earliest-opened arc, after which $m-2$
such down-steps remain with no pending condition. When $m \geq 2$, therefore
$w_m=w_{m-1}+w_{m-2}.$  Together with $w_0=2$ and $w_1=3$, this gives 
$w_m=F_{m+3}$ for all $m\geq0.$

There are $C_m$ possible paths $P$ of semilength $m$, and $UPD$ has $m+1$ up-steps. Its generating function is
therefore
$E(x)=
\sum_{m\geq0}C_mF_{m+3}x^{m+1} = xD(x)$, where 
$D(x)=\sum_{m\geq0}C_mF_{m+3}x^m.$

Every connected path consists of an initial up-step, a sequence
of these paths at height one, and a final down-step.
Each path starts and ends with one open arc and no pending
condition, so their contributions multiply. Therefore, the
generating function $H(x)$ for connected avoiders is
$$
H(x)=x\sum_{r\geq0}E(x)^r
=\frac{x}{1-xD(x)}.
$$

Every matching has a unique decomposition into connected
components. None of the forbidden patterns
can involve arcs from two different components, since all three
opening endpoints in such a pattern precede all three closing
endpoints. Thus, if 
$A(x)=\sum_{n\geq0}a_nx^n,$
we obtain
$\displaystyle A(x)=\frac{1}{1-H(x)}
=\frac{1-xD(x)}{1-x-xD(x)}.$
Equivalently,
$A(x)=1+xA(x)+xD(x)\bigl(A(x)-1\bigr).$
The empty matching and the single-arc matching give $a_0=a_1=1$.
Taking the coefficient of $x^n$ for $n\geq2$ gives the desired recurrence.
 \end{proof}

\begin{theorem}\label{thm:matching-initial-nesting}
For $n\geq h\geq1$, let $p_{n,h}$ be the number of connected
matchings on $[2n]$ avoiding $P_{13}=\{132,213,321\}$ 
and having its first $h$ arcs being mutually nested.
 That is, if the first $h$ arcs are
$(a_1,r_1),\ldots,(a_h,r_h)$ in opening order, then, 
we have $a_1<a_2<\cdots<a_h<r_h<\cdots<r_2<r_1.$
Define
$B_h(x)=\sum_{n\geq h}p_{n,h}x^n$ when  $h\geq1,$
and by convention, set $B_0(x)=1$.
Then, we have 
\begin{equation}\label{eq:matching-initial-nesting}
B_h(x) =
xB_{h-1}(x)+B_{h+1}(x) + x^h\sum_{j\geq1} \binom{h+j-2}{h-1}B_j(x).
\end{equation} 
In particular, $B_1(x)$ is the generating function for all
connected 
$P_{13}$-avoiding matchings.
\end{theorem}

\begin{proof}
Fix $h\geq1$, and let $M$ be a matching counted by $B_h(x)$.
Write its first $h$ arcs as
$A_i=(a_i,r_i) \>\> (1\leq i\leq h),$
where
$a_1<\cdots<a_h<r_h<\cdots<r_1.$
Since these are the first $h$ arcs and they are nested, the first
$h$ vertices are precisely their opening endpoints. Thus $a_i=i$
for $1\leq i\leq h$.

We distinguish three cases according to the endpoint at vertex
$h+1$ and, when it is an opening endpoint, the position of its
corresponding closing endpoint.

\medskip
\noindent
\textbf{Case 1: Vertex $h+1$ is a closing endpoint.}
The nesting forces this endpoint to close $A_h$, so
$A_h=(h,h+1).$

Suppose first that $h\geq2$. Deleting this adjacent arc leaves a
connected avoider whose first $h-1$ arcs are nested. Avoidance is
preserved by deletion, and deleting an arc between consecutive
vertices cannot introduce a separating cut in the remaining
matching. 
Further, start with a connected avoider whose first $h-1$ arcs
are nested. Insert a new adjacent arc immediately after their
opening endpoints. The resulting matching is connected, since
at least one of those initial arcs spans the inserted pair.
Its first $h$ arcs are nested.
Any forbidden occurrence containing the inserted arc would have
to use two arcs open across this adjacent pair. This is not
forbidden. Hence the insertion preserves avoidance and is
inverse to the deletion.
For $h=1$, the arc $(1,2)$ is a component by itself. Connectedness
therefore forces $M$ to be the single-arc matching. Thus this
case contributes $x$ when $h=1$, and $xB_{h-1}(x)$ when $h\geq2$.
With the convention $B_0(x)=1$, its contribution is uniformly
$xB_{h-1}(x).$

\begin{figure}[ht]
\begin{minipage}{.25\textwidth}
\begin{tikzpicture}[scale=0.6]
\node[] at (3.5,3.8) {\textcolor{white}{1} };
\draw[gray] (0.6,0)--(6.4,0);

\draw[blue!75!black,very thick]
    (1,0) arc (180:0:2.5);
\draw[blue!75!black,very thick]
    (2,0) arc (180:0:1.5);
\draw[blue!75!black,very thick]
    (3,0) arc (180:0:0.5);

\node[blue!75!black,above] at (3.5,2.5) {\footnotesize{$A_1$}};
\node[blue!75!black,above] at (3.5,1.5) {\footnotesize{$A_2$}};
\node[blue!75!black,above] at (3.5,0.5) {\footnotesize{$A_3$}};

\foreach \i in {1,...,6}{
    \fill (\i,0) circle (2pt);
    \node[below=5pt] at (\i,0) {$\i$};
}

\fill[red] (4,0) circle (3pt);

\end{tikzpicture}
\end{minipage}
\centering
\begin{minipage}{.31\textwidth}
\begin{tikzpicture}[scale=0.55]

\draw[gray] (0.6,0)--(8.4,0);

\draw[blue!75!black,very thick]
    (1,0) arc (180:0:3.5);
\draw[blue!75!black,very thick]
    (2,0) arc (180:0:2.5);
\draw[blue!75!black,very thick]
    (3,0) arc (180:0:1.5);

\draw[orange!85!black,very thick]
    (4,0) arc (180:0:0.5);

\node[blue!75!black,above] at (4.5,3.5) {\footnotesize{$A_1$}};
\node[blue!75!black,above] at (4.5,2.5) {\footnotesize{$A_2$}};
\node[blue!75!black,above] at (4.5,1.5) {\footnotesize{$A_3$}};
\node[orange!85!black,above] at (4.5,0.5) {\footnotesize{$T_1$}};

\foreach \i in {1,...,8}{
    \fill (\i,0) circle (2pt);
    \node[below=5pt] at (\i,0) {$\i$};
}

\fill[red] (4,0) circle (3pt);

\end{tikzpicture}
\end{minipage}
\begin{minipage}{.3\textwidth}
\begin{tikzpicture}[scale=0.55]

\draw[gray] (0.6,0)--(10.4,0);

\draw[blue!75!black,very thick]
    (1,0) arc (180:0:3.5);
\draw[blue!75!black,very thick]
    (2,0) arc (180:0:2.5);
\draw[blue!75!black,very thick]
    (3,0) arc (180:0:1);

\draw[orange!85!black,very thick]
    (4,0) arc (180:0:3);
\draw[orange!85!black,very thick]
    (6,0) arc (180:0:1.5);

\node[blue!75!black,above] at (4.5,3.5) {\footnotesize{$A_1$}};
\node[blue!75!black,above] at (4.5,2.5) {\footnotesize{$A_2$}};
\node[blue!75!black,above] at (3.8,1) {\footnotesize{$A_3$}};

\node[orange!85!black,above] at (7,3) {\footnotesize{$T_1$}};
\node[orange!85!black,above] at (7.9,1.5) {\footnotesize{$T_2$}};

\foreach \i in {1,...,10}{
    \fill (\i,0) circle (2pt);
    \node[below=5pt] at (\i,0) {$\i$};
}

\fill[red] (4,0) circle (3pt);

\end{tikzpicture}
\end{minipage}
\caption{ \small{\blue{Case 1:} When $h=3$, the vertex 
closes arc $A_3=(3,4)$, \blue{Case 2:} The arc $T_1=(4,5)$ is nested inside all the initial arcs, so $s_1<r_h$,} \blue{Case 3:} The arc $T_1$ crosses all
initial arcs $A_1,A_2,A_3$, while $T_1,T_2$
are mutually nested. Here $s_1>r_1$ }
\label{fig:P13}
\end{figure}

Suppose that vertex $h+1$ opens an arc $T_1=(c_1,s_1), c_1=h+1.$
If $r_{i+1}<s_1<r_i$ for some $1\leq i<h$, then $A_i,A_{i+1},T_1$ give a 
forbidden $132$. So, either
$s_1<r_h\text{ or } s_1>r_1.$

\noindent
\textbf{Case 2: When $s_1<r_h$. }
In this case $T_1$ is nested inside all the initial $h$ arcs.
The matchings in this case are exactly those whose first $h+1$
arcs are nested. Their contribution is
$B_{h+1}(x).$









\noindent
\textbf{Case 3: When $s_1>r_1$.}
In this case, $T_1$ crosses all the initial arcs $A_1,\ldots,A_h$.
Let $T_1,\ldots,T_j$ be the arcs, other than $A_1,\ldots,A_h$,
that open before $r_1$, listed in opening order.

We first show that these arcs are mutually nested and close after
$r_1$. Write $T_\ell=(c_\ell,s_\ell)$. For $\ell\geq2$, we have
$a_1<c_1<c_\ell<r_1<s_1.$
If $s_\ell<r_1$ or $s_\ell>s_1$, then
$A_1,T_1,T_\ell$ form the forbidden pattern $213$ or $321$,
respectively. Thus $r_1<s_\ell<s_1$. Moreover, if
$p<q$ and $s_p<s_q$, then $A_1,T_p,T_q$ form a forbidden
$321$. Hence
$r_1<s_j<s_{j-1}<\cdots<s_1,$
showing that $T_1,\ldots,T_j$ are mutually nested (see Figure \ref{fig:P13}).
Delete $A_1,\ldots,A_h$ and denote the resulting matching by $N$.
Clearly, $N$ avoids $P_{13}$ and its first $j$ arcs are nested.
It is also connected: $T_1$ crosses every cut before $s_1$,
while any separating cut after $s_1$ would also separate $M$,
as all deleted arcs close before $s_1$.
Thus $N$ is counted by $B_j(x)$.

Conversely, take a connected avoider $N$ whose first $j$ arcs
$T_1,\ldots,T_j$ are nested. Insert the opening endpoints of
$h$ new arcs $A_1,\ldots,A_h$ before all vertices of $N$,
followed by the opening endpoint $c_1$ of $T_1$.

Next, interleave the two ordered lists $r_h,r_{h-1},\ldots,r_2$ and $c_2,c_3,\ldots,c_j,$
preserving the order within each list. Place $r_1$ after this
interleaving, followed by the remaining endpoints of $N$
in their original order. The number of possible interleavings is $\displaystyle
\binom{h+j-2}{h-1}.$
The resulting matching is connected, since $A_1$ crosses every
cut before $r_1$, and connectedness thereafter follows from $N$.
It also avoids $P_{13}$. Indeed, any forbidden occurrence not
involving a new arc would already occur in $N$. Arcs of $N$
other than $T_1,\ldots,T_j$ open after all new arcs have closed,
so they cannot form a pattern with them. Among the remaining
arcs, triples consisting entirely of $A$-arcs or entirely of
$T$-arcs are nested. Mixed triples, whenever they form a
permutation pattern, have closing order $213$ or $132$,
both of which are allowed for $P_{13}$.
The construction is reversible, i.e.,  deleting $A_1,\ldots,A_h$
recovers $N$, and their closing endpoints determine the
interleaving uniquely. Therefore, for each $j\geq1$, the contribution is
$\displaystyle
x^h\binom{h+j-2}{h-1}B_j(x).$
Summing over $j\geq1$, Case 3 contributes
$\displaystyle
x^h\sum_{j\geq1}
\binom{h+j-2}{h-1}B_j(x).$










The three cases are disjoint and exhaustive. Adding their
contributions proves \eqref{eq:matching-initial-nesting}.
Finally, $B_j(x)\in x^j\mathbb{Z}[[x]]$, so only finitely many
summands contribute to any coefficient; the infinite sum is
therefore well-defined as a formal power series.
\end{proof}

\begin{theorem}\label{thm: p13-cont-frac}
    Let $A(x)$ be the generating function of the number of matchings on $[2n]$ avoiding $P_{13}=\{132, 213, 321\}$. Then $A(x)$ has the following expression as a continued fraction.
\begin{equation}
\label{eqn:cont_frac_last}
A(x) = 1 + \cfrac{q(1-q)}{d_1(q) - \cfrac{(1+q)q^3}{d_2(q) - \cfrac{q^5}{d_3(q) - \cfrac{q^7}{d_4(q) - \ddots}}}}.
\end{equation}
where $q = C(x) - 1$ and for $k \geq 1$, $d_k(q) = (1-q)^2 - 2q^k$.
\end{theorem}
\begin{proof}
Let $\displaystyle A(x)=\frac{1}{1-B_1(x)}$ and $\displaystyle 
F(x,z)=1+\sum_{h\ge 1} B_h(x)z^h.$
The recurrence \eqref{eq:matching-initial-nesting} simplifies to:
$\displaystyle (1-z+xz^2)(F(x,z)-1)=zB_1(x)-xz^2F\bigg(x,\frac{1}{1-xz}\bigg)$
Set $x=\frac{q}{(1+q)^2}$ so that $q=C(x)-1$. To get a clean three term recurrence,
we need one more substitution. Define $w$ to be a formal variable and set
$z(w)=\frac{(1+q)(1-w)}{1-qw}$. This gives us 
$\displaystyle 1-xz(w)=z(qw)$  and $\displaystyle 1-z(w)+xz(w)^2=\frac{w(1-q)^2}{(1-qw)^2}.$
With these, we get the functional equation
\begin{equation}\label{eq: new_subs}
q(1-z)^2G(q,qz)+z(1-q)^2G(q,z)=z(1-q)^2+(1+q)(1-z)(1-qz)B_1\bigg(\frac{q}{(1+q)^2}\bigg)
\end{equation}
where 
\begin{equation*}
    G(q,w):=F\bigg(\frac{q}{(1+q)^2}, \frac{(1+q)(1-w)}{1-qw} \bigg)=\sum_{n\ge 0} g_n(q)w^n.
\end{equation*}
We use \eqref{eq: new_subs} to build a three term recurrence as follows. Set $\overline{B}(q):=B_1\big(\frac{q}{(1+q)^2}\big)$. Extracting coefficients of $w^k$ from the functional equation
yields $q g_0 = (1+q)\overline{B}(q)$ at $w^0$, $(1-q)(1-2q) g_0 + q^2 g_1 = (1-q)^2$ at $w^1$, and the three-term recurrence
$$q(1-q)g_0 + d_2(q)g_1 + q^3g_2 = 0, \>\>
q^{k-1} g_{k-2} + d_k(q) g_{k-1} + q^{k+1} g_k = 0 \qquad (k \ge 3),
$$
where $d_k(q) = (1-q)^2 - 2q^k$.
Define the continued fraction tails 
$\displaystyle T_k(q) = d_k(q) - \frac{q^{2k+1}}{T_{k+1}(q)} \>\> (k \ge 2),
$
backward elimination gives $g_1 = -\dfrac{q(1-q)}{T_2(q)} g_0$. Substituting this into the equation for $w^1$ yields
\[
g_0 = \frac{1-q}{1 - 2q - \dfrac{q^3}{T_2(q)}} \implies \overline{B}(q) = \frac{q(1-q)}{(1+q)\left(1 - 2q - \dfrac{q^3}{T_2(q)}\right)}.
\]
Using both $A(x) = \dfrac{1}{1 - \overline{B}(q)}$ and the identity $(1+q)(1-2q) - q(1-q) = d_2(q)$ yields 
\eqref{eqn:cont_frac_last},
completing the proof.
\end{proof}

\section{Declaration of AI usage}
In our first version posted on \href{https://arxiv.org/abs/2609.08562v1}{ArXiv}, we left the classes $P_1$ and $P_{13}$ as open problems.  Subsequently, we used GPT6  
and found a recurrence for both of these which we 
have covered it in Section \ref{sec:ai_assist}. The recurrence were subsequently proved without the use of AI tools. Within two days, we got an email from Zhang Wenlin and Haobo \cite{zhang-email}
from the National University of Singapore (NUS)
who found the same recurrence for $P_1$.  On Oct 21, we got another 
email from Sam Zhang \cite{sam-zhang-email} (who is from Clemson University), 
who also gave the same recurrence for $P_1$ and the same solution
for $P_{13}$ and has more detailed information on these two 
problems.

\section{Acknowledgement}
The first author is supported by an Institute Post-doctoral
fellowship of Indian Institute of Technology, Bombay, India. 
The second author is
supported by an Institute Post-doctoral 
fellowship of Indian Institute of Technology, Madras, India.
The third author 
acknowledges support from project grant RD/0126-ANRF000-007 given by the 
Anusandhan National Research Foundation (ANRF). 
We are grateful to all of them for making this work possible.

\bibliographystyle{acm}

\end{document}